\documentclass[10pt]{amsart}
\usepackage{amsmath}
\usepackage{amsfonts}
\usepackage{amssymb}
\usepackage{graphicx}
\usepackage{amsthm,graphicx,color,yfonts}
\usepackage{pdfsync}
\usepackage{epstopdf}%
\usepackage{pifont}
\usepackage{bbm}

\usepackage[colorlinks=true]{hyperref}
\hypersetup{linkcolor=red,citecolor=blue,filecolor=dullmagenta,urlcolor=blue} 

\usepackage{wrapfig}
\usepackage{tikz}
\usetikzlibrary{arrows,calc,decorations.pathreplacing}
\definecolor{light-gray1}{gray}{0.90}
\definecolor{light-gray2}{gray}{0.80}
\definecolor{light-gray3}{gray}{0.60}

\newcommand{\wt}{\widetilde}

\newcommand{\vp}{\varphi}

\newcommand{\R} {\mathbb R}
\newcommand{\cuad}{{\sqcap\kern-.68em\sqcup}}

\newcommand{\ve}{\varepsilon}

\newcommand{\be}{\begin{equation}}
\newcommand{\ee}{\end{equation}}

\newcommand{\la}{\lambda}

\definecolor{darkgreen}{rgb}{0.2,0.7,0.1}

\newcommand{\sech}{\mathop{\mbox{\normalfont sech}}\nolimits}

\newcommand{\px}{\partial_x}
\newcommand{\pt}{\partial_t}
\newcommand{\nlop}{(1-\partial_x^2)^{-1}}

\newcommand{\al}{\alpha}
\newcommand{\bt}{\beta}

\def\bm{\left( \begin{array}{cc}}
\def\endm{\end{array}\right)}

\providecommand{\norm}[1]{\left\| #1 \right\|}

\newcommand{\ba}{\begin{equation*}}
\newcommand{\ea}{\begin{equation*}}
\newcommand{\bea}{\begin{eqnarray}}
\newcommand{\eea}{\end{eqnarray}}
\newcommand{\bee}{\begin{eqnarray*}}
\newcommand{\eee}{\end{eqnarray*}}
\newcommand{\ben}{\begin{enumerate}}
\newcommand{\een}{\end{enumerate}}

\numberwithin{equation}{section}

\newtheorem{theorem}{Theorem}[section]
\newtheorem*{theorem*}{Theorem}
\newtheorem{proposition}{Proposition}[section]
\newtheorem{corollary}{Corollary}[section]
\newtheorem{lemma}{Lemma}[section]
\newtheorem{definition}{Definition}[section]

\theoremstyle{remark}
\newtheorem{remark}{Remark}[section]

\title[Decay in abcd systems]{The scattering problem for the ABCD Boussinesq system in the energy space}

\author{Chulkwang Kwak}
\address{Facultad de Matem\'aticas, Pontificia Universidad Cat\'olica de Chile, Campus San Joaqu\'i­n. Avda. Vicu\~na Mackenna 4860, Santiago, Chile}
\email{chkwak@mat.uc.cl}
\thanks{C. K. is supported by FONDECYT Postdoctorate 2017 Proyect N$^{\circ}$ 3170067.}

\author{Claudio Mu\~noz}
\address{CNRS and Departamento de Ingenier\'{\i}a Matem\'atica and Centro
de Modelamiento Matem\'atico (UMI 2807 CNRS), Universidad de Chile, Casilla
170 Correo 3, Santiago, Chile.}
\email{cmunoz@dim.uchile.cl}
\thanks{C. M. work was partly funded by Chilean research grants FONDECYT  1150202, Fondo Basal CMM-Chile, MathAmSud EEQUADD and Millennium
Nucleus Center for Analysis of PDE NC130017.}

\author{Felipe Poblete}
\address{Universidad Austral de  Chile, Facultad de Ciencias, Instituto de Ciencias F\'isicas y Matem\'aticas,  Valdivia, CHILE.}
\email {felipe.poblete@uach.cl}
\thanks{F. P. is partially supported by Chilean research grant FONDECYT  1170466 and DID S-2017-43 (UACh).}

\author{Juan C. Pozo}
\address{Departamento de Matem\'atica y Estad\'istica, Facultad de Ciencias, Universidad de La Frontera, Casilla 54-D, Temuco, Chile}
\email{juan.pozo@ufrontera.cl}
\thanks{J. C. Pozo is partially supported by  Chilean research grant FONDECYT  11160295.}


\begin{document}

\begin{abstract}
The Boussinesq $abcd$ system is a 4-parameter set of equations posed in $\R_t\times\R_x$, originally derived by Bona, Chen and Saut \cite{BCS1,BCS2} as first order 2-wave approximations of the incompressible and irrotational, two dimensional water wave equations in the shallow water wave regime, in the spirit of the original Boussinesq derivation \cite{Bous}. Among many particular regimes, depending each of them in terms of the value of the parameters $(a,b,c,d)$ present in the equations, the \emph{generic} regime is characterized by the setting $b,d>0$ and $a,c<0$. If additionally $b=d$, the $abcd$ system is hamiltonian. The equations in this regime are globally well-posed in the energy space $H^1\times H^1$, provided one works with small solutions \cite{BCS2}. In this paper, we investigate decay and the scattering problem in this regime, which is characterized as having (quadratic) long-range nonlinearities, very weak linear decay $O(t^{-1/3})$ because of the one dimensional setting, and existence of non scattering solutions (solitary waves). We prove, among other results, that  for a sufficiently \emph{dispersive} $abcd$ systems (characterized only in terms of parameters $a, b$ and $c$), all small solutions must decay to zero, locally strongly in the energy space, in proper subset of the light cone $|x|\leq |t|$. We prove this result by constructing three suitable virial functionals in the spirit of works \cite{KMM1,KMM2}, and more precisely \cite{MPP} (valid for the simpler scalar ``good Boussinesq'' model), leading to global in time decay and control of all local $H^1\times H^1$ terms. No parity nor extra decay assumptions are needed to prove decay, only small solutions in the energy space.
\end{abstract}

\maketitle

\section{Introduction and Main Results}

\medskip

\subsection{Setting of the problem} This paper concerns with the study of nonlinear scattering and decay properties for the one-dimensional $abcd$ system introduced by Bona, Chen and Saut \cite{BCS1,BCS2}:
\begin{equation}\label{boussinesq_0}
\begin{cases}
(1- b\,\partial_x^2)\partial_t \eta  + \partial_x\!\left( a\, \partial_x^2 u +u + u \eta \right) =0, \quad (t,x)\in \R\times\R, \\
(1- d\,\partial_x^2)\partial_t u  + \partial_x\! \left( c\, \partial_x^2 \eta + \eta  + \frac12 u^2 \right) =0.
\end{cases}
\end{equation}
Here $u=u(t,x)$ and $\eta=\eta(t,x)$ are real-valued scalar functions. This equation was  originally derived by Bona, Chen and Saut \cite{BCS1} as a first order, one dimensional asymptotic regime model of the water waves equation, in the vein of the Boussinesq original derivation \cite{Bous}, but maintaining all possible equivalences between the involved physical variables, and taking into account the shallow water regime. The physical perturbation parameters under which the expansion is performed are
\[
\al := \frac Ah \ll 1, \quad \bt := \frac{h^2}{\ell^2} \ll 1, \quad \al \sim \bt, 
\]
and where $A$ and $\ell$ are typical waves amplitude and wavelength respectively, and $h$ is the constant depth. Equations \eqref{boussinesq_0} are part of a hierarchy of Boussinesq models, including second order systems, which were obtained in \cite{BCS1} and \cite{BCL}.

\medskip

A rigorous justification for this model \eqref{boussinesq_0} from the free surface Euler equations (as well as extensions to higher dimensions) was given by Bona, Colin and Lannes \cite{BCL}, see also Alvarez-Samaniego and Lannes \cite{ASL} for improved results. Since then, these models have been extensively studied in the literature, see e.g. \cite{BLS,BCL,LPS,Saut} and references therein for a detailed account of known results. This list is by no means exhaustive.

\medskip

The constants in \eqref{boussinesq_0} are not arbitrary and follow the conditions \cite{BCS1}
\be\label{Conds0}
a+b =\frac12\left(\theta^2-\frac13\right),\quad c+d =\frac12 (1-\theta^2)\geq 0,
\ee
for some $\theta\in [0,1].$ Moreover, $a+b+c+d = \frac13$ is independent of $\theta$. (This case is referred as the regime without surface tension $\tau\geq 0$, otherwise one has parameters $(a,b,c,d)$ such that $a+b+c+d = \frac13-\tau$.)

\medskip

In this paper, we will concentrate in the \emph{Hamiltonian generic} case \cite[p. 932]{BCS2}, namely the case where
\be\label{Conds2}
b,d >0, \quad a,~c<0, \quad b=d.
\ee
In the case \eqref{Conds2}, it is well-known \cite{BCS2} that \eqref{boussinesq_0} is globally well-posed in $H^s \times H^s$, $s\geq 1$, for small data, thanks to the preservation of the energy
\be\label{Energy}
E[u ,\eta ](t):= \frac12\int \left( -a (\partial_x u)^2 -c (\partial_x \eta)^2  + u^2+ \eta^2 + u^2\eta \right)(t,x)dx,
\ee
which is conserved by the flow. Since now, we will identify $H^1 \times H^1$ as the standard \emph{energy space} for \eqref{boussinesq_0}. 

\medskip

As for the low regularity Cauchy problem associated to \eqref{boussinesq_0} and its generalizations to higher dimensions, Saut et. al. \cite{SX,SWX} studied in great detail the long time existence problem by focusing in the small physical parameter $\ve$ appearing from the asymptotic expansions. They showed well-posedness (on a time interval of order $1/\ve$) for \eqref{boussinesq_0}. Previous results by Schonbek \cite{Schonbek} and Amick \cite{Amick} considered the case $a=c=d=0$, $b=\frac13$, a version of the original Boussinesq system, proving global well-posedness under a non cavitation condition, and via parabolic regularization. Linares, Pilod and Saut \cite{LPS} considered existence in higher dimensions for time of order $O(\ve^{-1/2})$, in the KdV-KdV regime $(b=d=0)$. Additional low-regularity well-posedness results can be found in the recent work by Burtea \cite{Burtea}. On the other hand, ill-posedness results and blow-up criteria (for the case $b=1$, $a=c=0$), are proved in \cite{CL}, following the ideas in Bona-Tzvetkov \cite{BT}.

\medskip

By considering the new stretching of variables $ u(t/\sqrt{b},x/\sqrt{b})$ and $\eta(t/\sqrt{b},x/\sqrt{b})$, we can assume $b=d=1$ and rewrite \eqref{boussinesq_0} as the slightly simplified model 
\begin{equation}\label{boussinesq}
\begin{cases}
(1- \partial_x^2)\partial_t \eta  + \partial_x \! \left( a\, \partial_x^2 u +u + u \eta \right) =0, \quad (t,x)\in \R\times\R, \\
(1- \partial_x^2)\partial_t u  + \partial_x \! \left( c\, \partial_x^2 \eta + \eta  + \frac12 u^2 \right) =0,
\end{cases}
\end{equation}
with new conditions (see \eqref{Conds0}-\eqref{Conds2})
\be\label{Conds}
a,c<0, \quad a +1 =\frac1{2b}\left(\theta^2-\frac13\right),\quad c+1 =\frac1{2b} (1-\theta^2)\geq 0,
\ee
for some $\theta\in [0,1].$  Note that in particular now $c>-1$, and equation \eqref{boussinesq} will be the main subject of this paper.

\medskip

Since \eqref{boussinesq} under \eqref{Conds} is globally well-posed in the energy space, at least for small solutions, knowing the asymptotic behavior of solutions for large time is of important relevance. However, being \eqref{boussinesq} a one-dimensional model, decay and scattering properties are often difficult to prove, in particular for nonlinearities (such as the ones present in \eqref{boussinesq} which are long range of not too flat at the origin, with respect to the expected scattering dynamics. Moreover, sometimes not even decay is allowed because of the existence of solitary waves which do not decay in the energy space. Therefore, nonlinear decay and scattering estimates must take into account these obstructions to decay.

\medskip

\subsection{Decay via Virial functionals} In a series of papers \cite{KMM1,KMM2}, the above problematics were attacked by using techniques directly related to the dynamics of dispersive solutions in the energy space. Decay properties of one dimensional wave models in the energy space were proved in presence of very mild assumptions and strong enemies, such as variable coefficients, presence of internal modes, and quadratic nonlinearities. The case studied in \cite{KMM2}, the one dimensional Klein-Gordon equation
\[
\partial_t^2 u  - \partial_x^2 u  + mu + f(u) =0, \quad m\in\R,
\]
although in part a consequence of the more involved results in \cite{KMM1} (see also \cite{Snelson}), is specially relevant to the contents of this paper, since it considers precisely perturbations of the vacuum solution, which is assumed to be stable for all time. In particular, the following result was proved:

\begin{theorem}[\cite{KMM2}]\label{ThmA}
Assume that $f$ is odd and $|f'(u)|\leq C |u|^{p-1}$, $|u|<1$. Then any globally defined small, odd solution in the energy space must locally decay to zero in space: for any sufficiently small $\ve$,
\be\label{wave}
\sup_{t\in \R} \|(u,u_t)(t)\|_{H^1\times L^2} <\ve \implies \lim_{t\to \pm\infty} \|(u,u_t)(t)\|_{(H^1\times L^2)(I)} =0,
\ee
for any compact interval $I$.
\end{theorem}

\begin{remark}
This result can be improved \cite{KMM3} by assuming $f$ not necessarily odd, but the flatness condition
\[
\exists C>0 \quad \hbox{s.t.} \quad \frac 12 u f(u) - F(u) \leq C |u|^6, \quad \forall  u\in\R, \quad F(u):= \int_0^u f(s)ds,
\]
is needed. This condition is satisfied for instance if $f$ is defocusing, or $f$ has focusing behavior, but $L^2$-supercritical at the origin. 
\end{remark}
\medskip

The essential point in the proof of \eqref{wave} was the following \emph{virial identity}: for any smooth, bounded $\psi=\psi(x)$, one has
\[
-\frac{d}{dt} \int u_t  \left(\psi u_x +\frac12 \psi' u \right) =  \int \psi' u_x^2 -\frac14 \int \psi''' u^2 - \int \psi'\left( \frac 12 u f(u) - F(u) \right), 
\]
which reveals that local control on $(u,u_x)$ can be obtained by proving the coercivity estimate
\[
  \int \psi' u_x^2 -\frac14 \int \psi''' u^2 - \int \psi'\left( \frac 12 u f(u) - F(u) \right) \gtrsim \int \psi' (u^2+ u_x^2).
\]
Precisely, this lower bound was obtained under the smallness and oddness assumptions stated in Theorem \ref{ThmA}. Roughly speaking, proving 
\[
-\frac{d}{dt} \int u_t  \left(\psi u_x +\frac12 \psi' u \right) \gtrsim \int \psi' (u^2+ u_x^2), 
\]
implies that the local $H^1$ norm of $u$ is converging to zero for large time, at least along a sequence of times.  (See also \cite{MM,MM1,MM2,MR} for preliminary introduction of virial identities to the study of the blow up dynamics in critical gKdV and NLS equations.) The reader may also consider the work by Lindblad and Tao \cite{Lin_Tao}, where similar  averaged decay estimates were obtained.

\medskip

In some sense, Theorem \ref{ThmA} can be recast as proving convergence to zero in time of the local energy norm, by-passing the standard linear decay estimates, that is to say, without proving convergence to a (modified or not) linear profile, which is well-known that decays.

\medskip

The versatility of the previous result, Theorem \ref{ThmA}, can be measured by the following two completely unrelated problems, for which very similar conclusions can be obtained, and even improving the conclusion in \eqref{wave}, in the sense that the interval $I$ in \eqref{wave} can be chosen \emph{growing as time evolves}. First, consider the quasilinear $1+1$ model known as the Born-Infeld model
\be\label{BI}
(1+(\partial_x u)^2)\partial_t^2 u +2 \partial_t u \, \partial_x u  \, \partial_{tx}^2 u -(1+(\partial_t u)^2)\partial_x^2 u =0, \quad u(t,x)\in\R.
\ee
Let $C_0>0$ be any large but fixed constant. In what follows, we will need the following, time-dependent, space interval:
\be\label{I(t)}
I(t):= \left(- \frac{C_0 |t|}{\log^2 |t|}, \frac{C_0|t|}{\log^2 |t|}\right), \quad |t|\geq 2.
\ee
Note that this interval is a slightly proper subset of the standard light cone $\{ |x|\leq |t| \}$, associated to wave-like equations. For equation \eqref{BI}, the following result was proved:
\begin{theorem}[\cite{AM1}]\label{ThmB}
There exists an $\ve>0$ such that if $(u,\partial_tu)$ solves \eqref{BI} with
\be\label{Cond_quasi}
\sup_{t\in \R}\|(u,\partial_t u)(t)\|_{H^s\times H^{s-1}}< \ve, \quad s>\frac32,
\ee
then one has, for any $C>0$ arbitrarily large and $I(t)$ defined in \eqref{I(t)},
\be\label{Conclusion_a}
\lim_{t \to\pm \infty}   \|(u,\partial_t u)(t)\|_{(H^1\times L^2)(I(t))} =0.
\ee
\end{theorem}
Note that condition \eqref{Cond_quasi} is needed in order to ensure local well-posedness for the quasilinear model \eqref{BI}. Also, it is needed for a good definition of virial identities, see \cite{AM1} for more details. Since $u(t,x):=\phi(x\pm t)$ is solution of \eqref{BI}, for any $\phi$ smooth, \eqref{Conclusion_a} establishes almost sharp decay estimates for the 1+1 dimensional BI model. 
  
\medskip

The second problem for which interesting conclusions can be stated is a classic one-wave fluid model. In \cite{MPP}, the authors considered extending the previous results to the case of the  \emph{good Boussinesq model} in 1+1 dimensions, 
\be\label{Good}
\begin{aligned}
\partial_t^2 u - \partial_x^2 u - \partial_x^4 u +\partial_x^2 f(u) =&~ 0, \quad  u(t,x)\in\R, \\
  f(0)=0, \quad |f'(s)| \lesssim  |s|^{p-1},  &~ |s|<1.
\end{aligned} 
\ee
This equation represents the simplest regularization of the originally derived (ill-posed), one-wave Boussinesq equation, which contains the fourth order term with positive sign: $+ \partial_x^4 u$. It is also a canonical model of shallow water waves as well as the Korteweg-de Vries (KdV) equation, see e.g. \cite{Bona}. Bona and Sachs \cite{Bona}, Linares  \cite{Linares}, and Liu \cite{Liu1,Liu2}, established that \eqref{boussinesq} is locally well-posed (and even globally well-posed for small data \cite{Bona,Linares}) in the standard energy space for $(u,\partial_t\partial_x^{-1} u) \in H^1\times L^2$. For this problem the following decay estimate was proved: 

\begin{theorem}[\cite{MPP}]\label{ThmC}
Consider \eqref{Good} in system form:
\be\label{Bous}
\begin{cases}
\partial_t u_1 = \partial_x u_2, \quad u_1 := u,\\
\partial_t u_2 = \partial_x (u_1- \partial_x^2 u_1-f(u_1)).
\end{cases}
\ee
There exists an $\ve>0$ such that if
\[
\|(u_1,u_2)(t=0)\|_{H^1\times L^2}< \ve,
\]
then one has, for any $C>0$ arbitrarily large and $I(t)$ defined in \eqref{I(t)},
\be\label{Conclusion_1}
\lim_{t \to\pm \infty}   \|(u_1,u_2)(t)\|_{(H^1\times L^2)(I(t))} =0.
\ee
\end{theorem} 
Note that this decay result holds even in the energy space, and even in the presence of solitary waves. In some sense, it can be referred as a \emph{nonlinear scattering} result, because even nonlinear waves are considered as possible cases in \eqref{Conclusion_1}. Previous results on scattering of small amplitude solutions of \eqref{Bous} were proved by  Liu \cite{Liu1}, Linares-Sialom \cite{LS}, and Cho-Ozawa \cite{Cho}. In these works, decay was proved in weighted spaces, and for $p\geq p_c$, a certain critical exponent for (subcritical or critical modified) scattering.

\subsection{Main results}  
In this paper we improve the results obtained in \cite{MPP} for the scalar good Boussinesq model, by considering now the physically more representative, and mathematically more involved $abcd$ system of Boussinesq equations \eqref{boussinesq}, under \eqref{Conds}. We prove, using well-chosen virial functionals, that all small globally defined $H^1\times H^1$ solutions must decay in time in time-growing intervals in space, slightly proper subsets of the light cone. In order to rigorously state such a result, we need first a quantitative notion of dispersion in terms of the parameters present in the equations. 

\begin{definition}[Dispersion-like parameters]\label{Dis_Par}
We say that $(a,b,c)$ satisfying \eqref{Conds} are \emph{dispersion-like parameters} if additionally
\be\label{dispersion_like}
3(a+c) + 2  < 8ac. 
\ee
\end{definition}

Condition \eqref{dispersion_like} formally says that $(a,c)$, being both negative, need to be sufficiently far from zero, depending on the remaining parameter $b>0$. Indeed, for values of $a$ and $c$ sufficiently close to zero,  \eqref{dispersion_like} is not satisfied. From the mathematical point of view, this condition can be understood as having \emph{sufficient dispersion} in \eqref{boussinesq} to allow decay to zero for both components of the flow. Additionally, along this paper we will prove the following sufficient conditions:

\begin{lemma}[Sufficient conditions for existence of dispersion-like parameters]\label{Sufficient}
Assume  $a,c<0$. Then the following holds:
\begin{enumerate}
\item[(i)]\label{(i)} The remaining conditions in \eqref{Conds} hold for a continuous segment of parameters $(a,c)$, if and only if $b>\frac16\sim 0.17$ (cf. Fig. \ref{Fig:0}). In other words, the $abcd$ Boussinesq system in the regime $a,c<0$ and $b=d>0$ makes sense only for $b>\frac16.$ 
\smallskip
\item[(ii)]\label{(ii)} Moreover, if in  addition $b>\frac29\sim 0.22$, then the dispersion-like condition \eqref{dispersion_like} holds for a continuous segment of nonempty parameters $(a,c)$, described in Fig. \ref{Fig:5}.
\end{enumerate}
\end{lemma}

Since condition \eqref{dispersion_like} is nonempty por $b$ ``large enough'', we are ready to state the main result of this paper.

\begin{theorem}\label{Thm1}
Let $(u,\eta)\in C(\R, H^1\times H^1)$ be a global, small solution of \eqref{boussinesq}-\eqref{Conds}, such that for some $\ve>0$ small
\be\label{Smallness}
\|(u,\eta)(t=0)\|_{H^1\times H^1}< \ve.
\ee
Assume additionally that $(a,c)$ are dispersion-like parameters as in \eqref{dispersion_like}. Then, for $I(t)$ as in \eqref{I(t)}, there is strong decay:
\be\label{Conclusion_0}
\lim_{t \to \pm\infty}   \|(u,\eta)(t)\|_{(H^1\times H^1)(I(t))} =0.
\ee
\end{theorem}

Some remarks are important to mention.

\begin{remark}
From \eqref{Smallness} and \eqref{Energy} one readily has $\sup_{t\in \R}  \|(u,\eta)(t)\|_{H^1\times H^1}  \lesssim \ve$, that is, zero is stable for $a,c<0$.
\end{remark}

\begin{remark}
Theorem \ref{Thm1} is, as far as we know, the first decay-in-time result for the original, one dimensional $abcd$ Boussinesq systems, which contains some of the most challenging ingredients in scattering theory. See also Mu\~noz and Rivas \cite{Munoz_Rivas}, where decay of order $O(t^{-1/3})$ is proved in weighted Sobolev spaces, for the case of a generalized $abcd$ system with $a=c$, $b=d$ and nonlinearities $\partial_x(u^p \eta)$ and $\partial_x(u^{p+1}/(p+1))$, and power $p\geq 5$. However, equations \eqref{boussinesq} (characterized by the exponent $p=1$) involve three important challenging issues: \emph{long range} (quadratic) nonlinearities, very weak linear dispersion $O(t^{-1/3})$ inherent to the one dimensional setting, and the existence of (non scattering) solitary waves, standard enemies to decay.  
\end{remark}

\begin{remark}[About the rate of decay]
Theorem \ref{Thm1} provides a mild rate of decay for solutions in the energy space: it turns out that (see \eqref{eq:virial1})
\be\label{eq:virial1_intro}
\int_{2}^\infty \!\! \int e^{-c_0 |x|} \left(u^2 + \eta^2 + (\partial_x u)^2+ (\partial_x \eta)^2\right)(t,x)dx\, dt  \lesssim_{c_0} \ve^2,
\ee
which reveals that local-in-space $H^1$ norms do integrate in time. This decay estimate may be compared with the formal rate of decay of the linear dynamics, that should be in general $O_{L^\infty}(t^{-1/3})$ \cite{Munoz_Rivas} (recall that not all $a,b,c$ lead to linear systems with the same rate of decay). In that sense, \eqref{eq:virial1_intro} reveals that a hidden improved decay mechanism is present in \eqref{boussinesq}. No better estimate can be obtained from our methods, as far as we understand, unless additional assumptions (on the decay of the initial data) are a priori granted. Additionally, let us emphasize that the linear rate of decay leads (at least formally) to the emergence of modified scattering (or ``supercritical scattering'' because of long range nonlinearities). See \cite{KMM1,KMM3} for a complete discussion on this topic.  
\end{remark}

\begin{remark}
The interval $I(t)$ can be made slightly more precise, for instance, 
\[
\tilde I(t):= \left(- \frac{C_0 |t|}{\log |t| \log^2( \log |t|)}, \frac{C_0|t|}{\log |t| \log^2 (\log |t|)}\right),  \quad |t|\geq 11,
\] 
or subsequent $\log (\cdots) $ improvements, are also valid. However, the interval with ends $\pm \frac{C_0 |t|}{\log |t|}$ cannot be reached by our results.
\end{remark}

\begin{remark}\label{Velocity_group}
Theorem \ref{Thm1} is also in formal agreement with the expected group velocity of \eqref{boussinesq} linear waves. 
It is well-known that  
the dispersion relation is given by
\[
w(k) = \frac{\pm |k| }{1+ k^2}(1-ak^2)^{1/2}(1-ck^2)^{1/2},
\]
which implies the following group velocity
\[
|w'(k)|=  \frac{|ack^6 +3ack^4 -(1+2a+2c)k^2 +1|}{(1+k^2)^2 (1-ak^2)^{1/2}(1-ck^2)^{1/2}}.
\]
It can be proved (see more details in Appendix \ref{A}), that $|w'(k)|>0$ for all $k\in\R $ if $b\geq\frac29$. Compare also with the conclusions in Corollary \ref{TH2} on solitary waves and their speeds.
\end{remark}

\begin{remark}
Based on some formal computations, we believe that condition \eqref{dispersion_like} is not sharp (but close to sharp), and Theorem \ref{Thm1} should hold for values of $b$ below $\frac29$, but with harder proofs.
\end{remark}

Theorem \ref{Thm1} in some sense corroborates the versatility of the Virial technique when proving decay, bypassing the proof of existence of an asymptotic (modified) linear profile by standard scattering techniques. These last methods need to exclude the existence of small solitary waves by using weighted norms, an assumption that is not necessary in our methods. Precisely, Theorem \ref{Thm1} is in concordance with the existence of solitary waves for \eqref{boussinesq}.

\medskip

The proof of Theorem \ref{Thm1}, as all previously stated results, is based in the introduction of suitable virial identities from which local dispersion in the energy space is clearly identified. Since \eqref{boussinesq} is a two-component system, any virial term must contain contributions from both variables $u$ and $\eta$. It turns out that the main virial term is given by
\[
\mathcal I(t) := \int \vp (u\eta + \px u \px\eta),
\]
where $\vp= \vp (x)$ is a smooth and bounded function. The introduction of the functional $\mathcal I$ is simple to explain, and it is directly inspired by the momentum
\[
P[u,\eta](t) := \int (u\eta + \partial_x u\partial_x\eta)(t,x)dx,
\]
which is conserved by the $H^1\times H^1$ flow \cite{BCS2}. However, in order to prove Theorem \ref{Thm1}, we will need to introduce two new functionals:
\[
\mathcal J(t) := \int \vp' \eta \px u, \quad \hbox{and}\quad \mathcal K(t) := \int \vp'\px\eta u.
\]
The functionals $\mathcal J$ and $\mathcal K$ can be understood as second order corrections to the main functional $\mathcal I(t)$, and allow us to cancel out several bad terms appearing from the  variation in time of $\mathcal I(t)$. It is important to notice that $\mathcal J(t) +\mathcal K(t)$ is somehow technically useless because it contains the total derivative $\partial_x(u\eta)$, so as far as we need it, we will  modify $\mathcal I(t) $ by using different linear combinations of $\mathcal J(t)$ and $\mathcal K(t)$, depending on the particular values of $a$ and $c$. 

\medskip

Precisely, condition \eqref{dispersion_like} ensures that a suitable linear combination of $\mathcal I(t)$, $\mathcal J(t)$ and $\mathcal K(t)$ is enough to show dispersion for any small solution in the energy space: we will prove that there are parameters $\al,\bt\in\R$ and $d_0>0$ such that 
\be\label{positivo}
\begin{aligned}
\frac{d}{dt}  \big(\mathcal I(t)  + & \al\mathcal J(t) +\bt \mathcal K(t) \big)  \\
& \gtrsim_{a,b,c,d_0} \int e^{-d_0|x|} (u^2 + \eta^2 + (\partial_x u)^2 +(\partial_x \eta)^2)(t,x)dx.
\end{aligned}
\ee
The choice of parameters in the linear combination $\mathcal I(t) +\al\mathcal J(t) +\bt \mathcal K(t)$ is not trivial and we need to change to well-known variables (see \cite{ElDika2005-2,ElDika_Martel} for its use in the BBM equation), that we refer as ``canonical'' because of its natural emergence in this type of problems. These new variables are introduced in order to show positivity of the bilinear form associated to the virial evolution, exactly as in \eqref{positivo}. Canonical variables are specially useful when dealing with local and nonlocal terms of the form
\[
\int e^{-d_0|x|} u^2 , \quad \int e^{-d_0|x|} u (1-\px^2)^{-1}u,
\]
and similar others, and which in principle have no clear connection or comparison. 

\medskip

The virial technique that we use in this paper is not new. It was inspired in the recent results obtained by the last three authors in \cite{MPP}, which are inspired in foundational works by Martel and Merle \cite{MM,MM1}, and Merle and Rapha\"el \cite{MR} for gKdV and NLS equations (see also \cite{CMPS} for another use of virial identities to show decay in the $H^1$ vicinity of the Zakharov-Kuznetsov solitary wave). An important improvement was the extension of virial identities of this type to the case of wave-like equations, see the works by Kowalczyk, Martel and the second author \cite{KMM1,KMM2}. Then, this approach was succesfuly extended to the case of the good Boussinesq system in \cite{MPP}. In this paper, we have followed some of the ideas in \cite{MPP}, especially the introduction of the functional $\mathcal I(t)$. However, the main difference between \cite{MPP} and this paper is that here we need to introduce two additional modifications of the main virial term $\mathcal I$ in order to show positivity (in most of the cases), and no Kato-smoothing effect to control higher derivatives is needed. We are also able to prove Theorem \ref{Thm1} without any assumption on the parity of the data, leading to a general result in the energy space.

\medskip

\subsection*{Solitary waves} In addition to long range nonlinearities and modified scattering, the existence of solitary waves is another enemy to decay (in the energy space essentially), and must be taken into account by the conclusions of Theorem \ref{Thm1}. In \cite{BCL}, the authors investigate the existence of solitary waves for \eqref{boussinesq} coming from ground states, namely solutions of the form
\be\label{Solitary}
u(t,x)= U_\omega(x-\omega t), \quad \eta(t,x)=N_\omega(x-\omega t), \quad \omega\in \R,
\ee
and where $(U_\omega,N_\omega)$ obey a variational characterization. Among other results, they prove the existence of (nonzero) ground states $(U_\omega, N_\omega)\in H^1\times H^1$ as long as 
\[
a,\, c<0, \quad b=d, \quad |\omega|<\omega_0,\quad \omega_0 := \begin{cases} \min\{1,\sqrt{ac}\} , \quad  b\neq 0, \\ 1, \quad b=0. \end{cases}
\]
(Note that the speed $\omega$ never reaches the speed of light $=1$.) The construction of $(U_\omega,N_\omega)$ in \cite{BCL} is involved, and based in stationary, elliptic techniques, and minimization techniques in well-chosen admissible manifolds. See also \cite{CNS1,CNS2,Olivera} for further results on the existence of solitary waves for \eqref{boussinesq}.

\medskip

Although not explicit in general, $(U_\omega, N_\omega)$ are sometimes explicit. For instance (\cite{HSS}), for $\omega=0$ and $a=c<0$, 
\[
\begin{aligned}
(U_0, N_0)(x):=&~ {} \left(\sqrt{2}\, Q \left(\frac{x}{\sqrt{|a|}}\right) , -Q \left(\frac{x}{\sqrt{|a|}}\right) \right), \\
 Q(x): =&~ {} \frac{3}{2\cosh^2 \big(\frac x2\big)}, \qquad Q>0 ~\hbox{ solution of }~ Q''-Q + Q^2=0,  
\end{aligned}
\]
is a standing wave solution of \eqref{boussinesq} (see \cite{HSS} for other explicit solitary wave solutions, as well as the study of their linear and nonlinear stability). This simple fact also reveals that zero speed solutions are somehow ``too large'' to be considered by the hypotheses of Theorem \ref{Thm1}, and exactly this is the case. Indeed, Theorem \ref{Thm1} gives a \emph{dynamical} proof of nonexistence of any small solitary waves with zero speed. 

\begin{corollary}[Nonexistence of non scattering waves]\label{TH2}
Under the assumptions of Theorem \ref{Thm1}, the following hold:
\ben
\item The $abcd$ Boussinesq system \eqref{boussinesq} has no small, zero-speed solitary waves in the form \eqref{Solitary}. 
\smallskip
\item Any small solitary wave of nonzero speed $\omega$ exits the interval $I(t)$ in \eqref{I(t)} in finite time, no matter how small the speed $\omega$ is.
\smallskip
\item No standing small breather solutions exists for \eqref{boussinesq}.
\een
\end{corollary}

By standing breather solution we mean a solution $(u,\eta) \in H^1\times H^1$ of \eqref{boussinesq} which is global and periodic in time. Several dispersive models do have stable and unstable breathers, and their existence are important counterexamples to asymptotic stability of the vacuum solution. Corollary \ref{TH2} discards the existence of such solutions near any bounded neighborhood of zero. See e.g. \cite{AM,AM2,AMP1,AMP2,Munoz,Alejo} and references therein for further details on breather solutions.

\medskip

Finally, note that Corollary \ref{TH2} shows that Theorem \ref{Thm1} is almost sharp (except by a $\log^2 t$ loss), in the sense that solitary waves \eqref{Solitary} move along lines $x = \omega t + const.$, and they do not decay in the energy space. 

\subsection*{Organization of this paper} 
This paper is organized as follows: Section \ref{2} contains some preliminary results needed along this paper, as well as the proof of Lemma \ref{Sufficient}, first item.  Section \ref{VIRIAL} deals with the Virial identities needed for the proof of Theorem \ref{Thm1}. Section \ref{sec:general virial} contains the study of the positivity of the bilinear form appearing from the virial dynamics, and the proof of Lemma \ref{Sufficient}, second item. The purpose of Section \ref{5} is to obtain integrability in time of the local $H^1\times H^1$ norm of the solution, the main consequence of the virial identity. Section \ref{ENERGY} contains energy estimates needed to improve virial estimates from previous sections. Finally, in Section \ref{7} we prove Theorem \ref{Thm1}.

\subsection*{Acknowledgments} We would like to thank Jerry L. Bona for introducing to us the $abcd$ system and explaining to us the Boussinesq models, and Jean-Claude Saut for some useful criticisms, as well as explaining to us the main results for local and global existence. Part of this work was done while the second author was visiting U. Austral (Valdivia, Chile), and the first and second author were visiting UNICAMP at Campinas (Brazil) because of \emph{Third Workshop on Nonlinear Dispersive Equations}, places where part of this work was done. The authors acknowledge the warm hospitality of both institutions. Part of this work was also complete while first two authors were part of \emph{EEQUADD MathAmSud Workshop 2017}, held at DIM-CMM U. of Chile. Finally, we thank Miguel Alejo for a careful proofreading of a first version of this manuscript.

\bigskip

\section{Preliminaries}\label{2}

\medskip

This section describes some auxiliary results that we will need in the following sections. 

\subsection{Numerology in the $abcd$ Boussinesq system}\label{abc Conds} Consider the parameters $(a,b,c,d)$ satisfying \eqref{Conds0} and  \eqref{Conds2} only (note that we are not assuming \eqref{Conds}). The purpose of this paragraph is to better understand these conditions, in order to justify some computations in Section \ref{sec:general virial}.

\medskip  

First of all, note that \eqref{Conds2} and \eqref{Conds0} imply that 
\be\label{eq:par_cond3}
b=\frac16 -\frac12(a+c),
\ee
and then 
\[
b>\frac16.
\]
This will be the starting point of the following reasoning. Using \eqref{Conds0} we also obtain
\begin{equation}\label{eq:par_cond4}
b = \frac12\left(\theta^2 - \frac13\right) - a, \quad \theta \in [0,1]  ,
\end{equation}
and
\begin{equation}\label{eq:par_cond5}
b = \frac12\left(1-\theta^2 \right) - c, \quad \theta \in [0,1].
\end{equation} 
Fix $b \in \R$. Let us define sets of pairs $(a,c)$ satisfying the following conditions: 
\be\label{Bj}
\begin{aligned}
\mathcal{B}_0 := & \left\{ (a,c) \in \R^2 : a < 0, \; c<0\right\}, \\
\mathcal{B}_1(b) :=& \left\{ (a,c) \in \R^2 : a+c = \frac13 - 2b\right\},\\
\mathcal{B}_2(b) := & \left\{ (a,c) \in \R^2 : -b -\frac16 \le a \le - b + \frac13 \right\}, \\
\mathcal{B}_3(b) := & \left\{ (a,c) \in \R^2 : -b \le c \le -b +\frac12 \right\}.
\end{aligned}
\ee
Note that $\mathcal{B}_1(b)$, $\mathcal{B}_2(b)$ and $\mathcal{B}_3(b)$ reflect the conditions \eqref{eq:par_cond3}, \eqref{eq:par_cond4} and \eqref{eq:par_cond5}, respectively (see Fig. \ref{Fig:0} for details). One can see that the set $\mathcal{B}_1(b)$ describes the line on the $(a,c)$-plane passing through the points $\big(-b-\frac16 ,  -b + \frac12\big)$ and $\big( -b +\frac13 , -b \big)$, which means $\mathcal{B}_1(b) \cap \mathcal{B}_2(b) \cap \mathcal{B}_3(b) \neq \emptyset$ for all $b \in \R$. Moreover, 
\[
\mathcal{B}_0 \cap \mathcal{B}_1(b) \cap \mathcal{B}_2(b) \cap \mathcal{B}_3(b) \neq \emptyset
\]
when $b > \frac16$. We resume these findings in the following result.

\begin{lemma}\label{1o6}
There exist $(a,b,c,d)$ such that \eqref{Conds0} and  \eqref{Conds2} are satisfied if and only if $b>\frac16.$ 
\end{lemma}

Note that this result proves Lemma \ref{Sufficient}, item (i).

\begin{figure}[h!]
\begin{center}
\begin{tikzpicture}[scale=0.8]
\filldraw[thick, color=lightgray!30] (-1,1.5)--(5.2,1.5) -- (5.2,5) --(-1,5) -- (-1,1.5);
\filldraw[thick, color=lightgray!10] (-1,-1)--(-1,4) -- (4,4) --(4,-1) -- (-1,-1);
\filldraw[thick, color=lightgray!60] (0.5,-1)--(0.5,5.3) -- (3.2,5.3) --(3.2,-1) -- (0.5,-1);
\draw[thick, color=black] (1.25,4) -- (3.2,1.5);
\draw[thick,dashed] (0.5,5) -- (1.25,4);
\draw[thick,dashed] (3.2,1.5) -- (5.2,-1);
\draw[thick,dashed] (0.5,-1)--(0.5,5.3);
\draw[thick,dashed] (3.2,-1)--(3.2,5.3);
\draw[thick,dashed] (-1,1.5)--(5.2,1.5);
\draw[thick,dashed] (-1,5)--(5.2,5);
\draw[->] (-1,4) -- (5.2,4) node[below] {$a$};
\draw[->] (4,-1) -- (4,5.3) node[right] {$c$};
\node at (1.9,-0.7){$ \mathcal{B}_2(b)$};
\node at (3,4){$\bullet$};
\node at (2.7,4.4){$-\frac16$};
\node at (1.5,4){$\bullet$};
\node at (1.5,4.4){$-b$};
\node at (0.5,4){$\bullet$};
\node at (-0.3,4.4){$-b -\frac16$};
\node at (4,1.5){$\bullet$};
\node at (4.3,1.2){$-b$};
\node at (4,5){$\bullet$};
\node at (4.7,4.6){$-b+\frac12$};
\node at (4.7,3.3){$ \mathcal{B}_3(b)$};
\node at (1.5,2.7){$ \mathcal{B}_1(b)$};
\node at (-0.7,3.5){$\mathcal{B}_0$};
\end{tikzpicture}
\qquad 
\begin{tikzpicture}[scale=0.8]
\filldraw[thick, color=lightgray!60] (0,4.7)--(0,2) -- (4/3,1) --(4,3) -- (4,4.7) -- (0,4.7);
\draw[thick,dashed] (4,-1) -- (4,3);
\draw[thick] (4,3) -- (4,5);
\draw[thick,dashed] (0,0) -- (4,3);
\draw[thick,dashed] (0,2)--(8/3,0);
\draw[thick] (0,2) -- (0,5);
\draw[->] (-1,0) -- (5,0) node[below] {$\nu$};
\draw[->] (0,-1) -- (0,5) node[right] {$b$};
\node at (0,0){$\bullet$};
\node at (4,0){$\bullet$};
\node at (4.3,-0.4){$1$};
\node at (4/3,0){$\bullet$};
\node at (8/3,-0.4){$\frac23$};
\node at (4/3,-0.4){$\frac13$};
\node at (8/3,0){$\bullet$};
\node at (0,2){$\bullet$};
\node at (-0.3,2){$\frac 13$};
\node at (0,1){$\bullet$};
\node at (-0.3,1){$\frac16$};
\node at (0,3){$\bullet$};
\node at (-0.3,3){$\frac12$};
\node at (4.8,3){$b=\frac12\nu$};
\node at (2,3){$ \mathcal{R}_0$};
\end{tikzpicture}
\end{center}
\caption{(\emph{Left}). The set $\mathcal{B}_0 \cap \mathcal{B}_1(b) \cap \mathcal{B}_2(b) \cap \mathcal{B}_3(b) $ depending on the parameter $b$, in the case where $b>\frac16$. The {\bf continuous} segment of the line $\mathcal{B}_1(b)$ included in $\mathcal{B}_0 \cap \mathcal{B}_2(b) \cap \mathcal{B}_3(b) $ corresponds to the admissible set of parameters $(a,c)$ (depending on $b$) for which the $abcd$ Boussinesq system makes physical sense. (\emph{Right}). A new representation (in bold) of the set $\mathcal R_0$ defined in \eqref{P0} in terms of free parameters $(\nu,b)$, see \eqref{eq:parameter set0}. Note that each point has associated values $(a,c)$ via formula  \eqref{eq:parameter set0}, and the set of admissible values makes sense only if $b>\frac16$. Note that at $(\nu,b)=(\frac13,\frac16)$, one has $(a,c)=(0,0)$.}\label{Fig:0}
\end{figure}
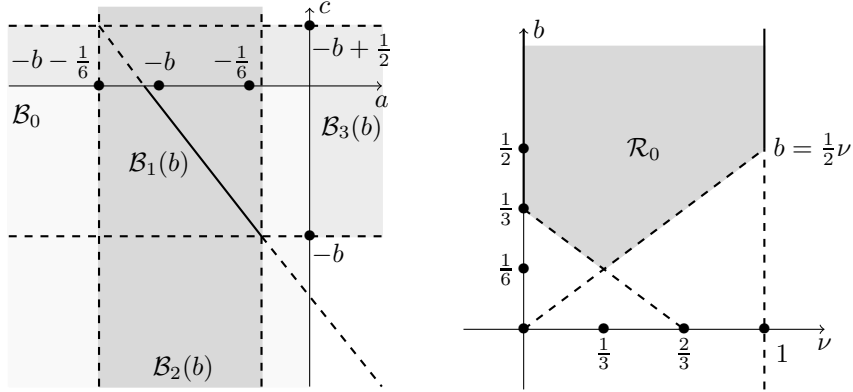
 
\medskip

\begin{remark}
An alternative form for expressing the set $\mathcal{P}_0(b):= \mathcal{B}_0 \cap \mathcal{B}_1(b) \cap \mathcal{B}_2(b) \cap \mathcal{B}_3(b)$ above described is the following 2-parameter description: the set of points $(a,b,c)$ such that 
\begin{equation}\label{eq:parameter set0}
(a,b,c) = \left(-\frac{\nu}{2} + \frac13 -b,\; b,\; \frac{\nu}{2} - b \right), \quad \nu \in [0,1] \cap \left(\frac23 - 2b, 2b \right).
\end{equation}
Finally, we introduce the sets
\be\label{P0}
\mathcal{P}_0 := \bigcup_{b>1/6} \mathcal{P}_0(b), \quad \mathcal{R}_0:=  \left\{ (\nu,b) ~: ~   \nu \in [0,1] \cap \left(\frac23 - 2b, 2b \right), ~b>\frac16\right\}.
\ee
See Fig. \ref{Fig:0}, right panel, for more details.
\end{remark}

\subsection{Canonical variables} It turns out that virial identities for \eqref{boussinesq} are not well-suited for the variables $(u,\eta)$, but instead  better suited for natural canonical variables that appear in models involving the nonlocal operator $(1-\px^2)^{-1}$.

\begin{definition}[Canonical variable]\label{Can_Var}
Let $u=u(x) \in L^2$ be a fixed function. We say that $f$ is canonical variable for $u$ if $f$ uniquely solves the equation
\be\label{Canonical}
f- f_{xx} = u, \quad f\in H^2(\R).
\ee
In this case, we simply denote $f=  (1-\partial_x^2)^{-1} u.$
\end{definition}

Canonical variables are standard in equations where the operator $(1-\partial_x^2)^{-1}$ appears; one of the well-known example is given by the Benjamin-Bona-Mahoney BBM equation, see e.g. \cite{ElDika2005-2,ElDika_Martel}.

\medskip

\begin{lemma}[Equivalence of local $L^2$ norms]\label{lem:L2 comparable}
Let $f$ be a canonical variable for $u \in L^2$, as introduced in Definition \ref{Can_Var}. Let $\phi$ be a smooth, bounded positive weight satisfying $|\phi''| \le \lambda \phi$ for some small but fixed $0 < \lambda \ll1$. Then, for any $a_1,a_2,a_3 > 0$, there  exist $c_1, C_1 >0$, depending on $a_j$ ($j=1,2,3$) and $\lambda >0$, such that 
\begin{equation}\label{eq:L2_est}
c_1  \int \phi \, u^2 \le \int \phi\left(a_1f^2+a_2f_x^2+a_3f_{xx}^2\right) \le C_1 \int \phi \, u^2.
\end{equation}
\end{lemma}

\begin{proof}
A direct calculation using \eqref{Canonical} gives the identity
\begin{equation}\label{eq:L2-1}
\int \phi u^2 = \int \phi\left(f^2 + 2f_x^2 +  f_{xx}^2\right) - \int \phi''f^2.
\end{equation}
From the property of $\phi$, the right-hand side of \eqref{eq:L2-1} is bounded by
\[
\int \phi\left((1+\lambda)f^2 + 2f_x^2 +  f_{xx}^2\right),
\]
for some $0 < \lambda \ll1$. Let $\tilde{a} := \min(a_1,a_2,a_3)$ and $c_1 := \frac{\tilde{a}}{2}$. Then, we have
\[
\begin{aligned}
\int \phi u^2 &\le \int \phi\left((1+\lambda)f^2 + 2f_x^2 +  f_{xx}^2\right)\\
&\le c_1^{-1}\int \phi\left(a_1f^2+a_2f_x^2+a_3f_{xx}^2\right).
\end{aligned}
\]
For the reverse part, we have from \eqref{eq:L2-1} that
\[
\int \phi u^2 \ge \int \phi\left((1-\lambda)f^2 + 2f_x^2 +  f_{xx}^2\right),
\]
for some $0 < \lambda \ll1$.  Let $\tilde{A} := \max(a_1,a_2,a_3)$ and $C_1 := \frac{\tilde{A}}{1-\lambda}$. Then, a direct calculation yields
\[\begin{aligned}
\int \phi u^2 &\ge \int \phi\left((1-\lambda)f^2 + 2f_x^2 +  f_{xx}^2\right)\\
&\ge C_1^{-1}\int \phi\left(a_1f^2+a_2f_x^2+a_3f_{xx}^2\right).
\end{aligned}
\]
The proof is complete.
\end{proof}
The same argument as in the proof of Lemma \ref{lem:L2 comparable} yields the following result.

\begin{lemma}[Equivalence of local $H^1$ norms]\label{lem:H1 comparable}
Let $f$ be a canonical variable for $u \in H^1$ defined as in Definition \ref{Can_Var}. Let $\phi$ be a smooth, bounded positive weight satisfying $|\phi''| \le \lambda \phi$ for some small but fixed $0 < \lambda \ll1$. Then, for any $d_1,d_2,d_3 > 0$, there  exist $c_1, C_1 >0$ depending on $d_j$, $j=1,2,3$, and $\lambda >0$ such that 
\begin{equation}\label{eq:H1_est}
c_2 \int \phi u_x^2 \le  \int \phi\left(d_1f_x^2+d_2f_{xx}^2+d_3f_{xxx}^2\right) \le C_2 \int \phi u_x^2.
\end{equation}
\end{lemma}

\medskip

\subsection{Comparison principle in terms of canonical variables}
The following results are well-known in the literature, see El Dika \cite{ElDika2005-2} for further details and proofs.

\begin{lemma}[\cite{ElDika2005-2}]\label{Dika1}
The operator $ (1-\px^2)^{-1}$ satisfies the following comparison principle: for any $u,v\in H^1$,
\begin{equation}\label{eq:inverse op1}
v \le w \quad  \Longrightarrow \quad (1-\px^2)^{-1} v \le (1-\px^2)^{-1} w.
\end{equation}
\end{lemma}
\eqref{eq:inverse op1} is a consequence of the fact that $(1-\px^2)^{-1} u\sim e^{-|x|} \star u $. Next result is a consequence of Lemma \ref{Dika1} and the Sobolev embedding $u \in H^1 \implies u\in L^\infty$, see the proof of Lemma 2.1 in \cite{ElDika2005-2}.

\begin{lemma}[\cite{ElDika2005-2}]\label{lem:nonlinear1}
Suppose that $\phi =\phi(x)$ is such that 
\begin{equation}\label{eq:inverse op2}
(1-\px^2)^{-1}\phi(x) \lesssim \phi(x), \quad x\in \R,
\end{equation}
for $\phi(x) > 0$ satisfying $|\phi^{(n)}(x)| \lesssim \phi(x)$, $n \ge 0$. Then, for $v,w,h \in H^1$, we have
\begin{equation}\label{eq:nonlinear1-1}
\int \phi^{(n)} v (1-\px^2)^{-1}(wh)_x ~\lesssim ~ \norm{v}_{H^1} \int \phi (w^2 + w_x^2 +h^2 + h_x^2),
\end{equation}
and
\begin{equation}\label{eq:nonlinear1-2}
\int\phi^{(n)} v (1-\px^2)^{-1}(wh) ~\lesssim ~\norm{v}_{H^1} \int \phi(w^2 +h^2).
\end{equation}
\end{lemma}

We will also need modified versions of Lemma \ref{lem:nonlinear1}, which will be useful when estimating nonlinear terms in the energy estimate (Section \ref{ENERGY}).

\begin{lemma}\label{lem:nonlinear2}
Under the same condition as in Lemma \ref{lem:nonlinear1}, we have
\begin{equation}\label{eq:nonlinear2-1}
\int (\phi v_x)_x (1-\px^2)^{-1}(wh) \lesssim \norm{v}_{H^1} \int \phi (w^2 + w_x^2 +h^2 + h_x^2).
\end{equation}
\end{lemma}

\begin{proof}
The product rule gives
\begin{equation}\label{eq:product}
\phi v_x = (\phi v)_x - \phi'v \quad \Longrightarrow \quad  (\phi v_x)_x = (\phi v)_{xx} - (\phi'v)_x.
\end{equation}
\eqref{eq:product} and the integration by parts lead to
\[\begin{aligned}
\mbox{LHS of }\eqref{eq:nonlinear2-1} =&~{} \int \left( (\phi v)_{xx} - (\phi'v)_x \right)(1-\px^2)^{-1}(wh) \\
=&~{}-\int\phi vwh + \int\phi v(1-\px^2)^{-1}(wh)\\ 
&+ \int \phi'(x)v(1-\px^2)^{-1}(w_xh+wh_x).
\end{aligned}\]
Thus, the Cauchy-Schwarz inequality and Lemma \ref{lem:nonlinear1} proves \eqref{eq:nonlinear2-1}.
\end{proof}

\begin{lemma}\label{lem:nonlinear3}
Under the same condition as in Lemma \ref{lem:nonlinear1}, we have
\begin{equation}\label{eq:nonlinear3-1}
\int \phi v_x (1-\px^2)^{-1} (wh)_x \lesssim \norm{v}_{H^1} \int \phi (w^2 + w_x^2 +h^2 + h_x^2).
\end{equation}
\end{lemma}

\begin{proof}
The left-hand side of \eqref{eq:product} and the integration by parts allow us to get
\[\begin{aligned}
\mbox{LHS of }\eqref{eq:nonlinear3-1} =&~ {}-\int \phi' v (1-\px^2)^{-1} (w_xh + wh_x) \\
&  - \int \phi v (1-\px^2)^{-1} (wh)_{xx} \\
=&~ {}-\int \phi' v (1-\px^2)^{-1} (w_xh + wh_x) \\
&  + \int \phi vwh - \int \phi v (1-\px^2)^{-1}(wh) \\
\end{aligned}\]
Thus, the Cauchy-Schwarz inequality and Lemma \ref{lem:nonlinear1} proves \eqref{eq:nonlinear3-1}.
\end{proof}

\bigskip
\section{Virial functionals}\label{VIRIAL}

\medskip

This Section is devoted to the introduction and study of three virial functionals, and their behavior under the $H^1\times H^1$ flow.  Let $\vp=\vp(x)$ be a smooth, bounded weight function, to be chosen later. For each $t\in\R$, we consider the following functionals for some $\vp$ (to be chosen later):
\be\label{I}
\mathcal I(t) := \int \vp(x)(u\eta + \px u \px\eta)(t,x)dx,
\ee
\be\label{J}
\mathcal J(t) := \int \vp'(x)(\eta \px u)(t,x)dx,
\ee
and
\be\label{K}
\mathcal K(t) := \int \vp'(x)(\px\eta u)(t,x)dx.
\ee
Clearly each functional above is well-defined for $H^1\times H^1$ functions, as long as the pair $(u,\eta)(t=0)$ is small in the energy space.

\subsection{Virial functional $\mathcal I(t)$} Using \eqref{boussinesq} and integration by parts, we have the following result.

\begin{lemma}\label{Virial_bous}
For any $t\in \R$,
\be\label{Virial0}
\begin{aligned}
\frac{d}{dt} \mathcal I(t) = &~  {}  -\frac{a}2 \int  \varphi' u_x^2-\frac{c}2 \int  \varphi' \eta_x^2 \\
& ~ {} - \left(a+\frac12\right)   \int  \varphi' u^2 -\left( c+ \frac12 \right)  \int  \varphi' \eta^2 \\
& ~ {} + (1+ a)\int \varphi' u (1-\partial_x^2)^{-1} u  + (1+ c)\int \varphi' \eta (1-\partial_x^2)^{-1} \eta  \\
&~ {}     -\frac12 \int  \varphi' u^2 \eta   + \int \varphi' u (1-\partial_x^2)^{-1}\left(u \eta \right)  + \frac12\int \varphi' \eta (1-\partial_x^2)^{-1}\left(u^2\right) .
\end{aligned}
\ee
\end{lemma}

\begin{proof}
We compute:
\[
\begin{aligned}
\frac{d}{dt} \mathcal I(t) =&~  \int \varphi (\eta_t u + \eta u_t + u_{tx} \eta_x + u_x \eta_{tx} )\\
=&~  \int \varphi (\eta_t -\eta_{txx}) u + \int  \varphi (u_t  -u_{txx}) \eta -\int \varphi' \eta u_{tx} -\int \varphi' u \eta_{tx}.
\end{aligned}
\]
Replacing \eqref{boussinesq}, and integrating by parts, we get
\[
\begin{aligned}
\frac{d}{dt} \mathcal I(t) =&~  \int (\varphi  u )_x (a u_{xx} +u +u\eta ) + \int  (\varphi \eta)_x \left(c\eta_{xx} +\eta + \frac12 u^2\right)\\
&~ {} + \int (\varphi' \eta)_x u_{t} + \int (\varphi' u)_x \eta_{t} \\
= :&~ I_1 + I_2 +I_3 +I_4.
\end{aligned}
\]
First of all, using \eqref{boussinesq},
\[
\begin{aligned}
I_3= &~  \int (\varphi' \eta)_x u_{t} =  \int (\varphi' \eta)_{xx} (1-\partial_x^2)^{-1}\left(c\eta_{xx} +\eta +\frac12 u^2\right) \\
=&~   \int \left( (\varphi' \eta)_{xx} - \varphi' \eta \right) (1-\partial_x^2)^{-1}\left(c\eta_{xx} +\eta +\frac12 u^2\right) \\
&~ {}  + \int \varphi' \eta (1-\partial_x^2)^{-1}\left(c\eta_{xx} +\eta +\frac12 u^2\right) \\
=&~  - \int  \varphi' \eta\left(c\eta_{xx} +\eta +\frac12 u^2\right) \\
&~ {}  + \int \varphi' \eta (1-\partial_x^2)^{-1}\left(c\eta_{xx} +\eta +\frac12 u^2\right) .
\end{aligned}
\]
Therefore,
\[
I_2 + I_3 = \int  \varphi \eta_x \left(c\eta_{xx} +\eta + \frac12 u^2\right)+ \int \varphi' \eta (1-\partial_x^2)^{-1}\left(c\eta_{xx} +\eta +\frac12 u^2\right).
\]
Similarly,
\[
\begin{aligned}
I_4  =&~  - \int  \varphi' u \left(au_{xx} +u +u \eta \right) \\
&~ {}  + \int \varphi' u (1-\partial_x^2)^{-1}\left(a u_{xx} + u + u \eta\right),
\end{aligned}
\]
and
\[
I_1 + I_4 = \int  \varphi u_x \left(au_{xx} +u +u \eta \right)+ \int \varphi' u (1-\partial_x^2)^{-1}\left(a u_{xx} +u +u \eta \right).
\]
We conclude that 
\[
\begin{aligned}
\frac{d}{dt}\mathcal I(t) = &~ \int  \varphi u_x \left(au_{xx} +u +u \eta \right)+ \int \varphi' u (1-\partial_x^2)^{-1}\left(a u_{xx} +u +u \eta \right)\\
&~ {} + \int  \varphi \eta_x \left(c\eta_{xx} +\eta + \frac12 u^2\right)+ \int \varphi' \eta (1-\partial_x^2)^{-1}\left(c\eta_{xx} +\eta +\frac12 u^2\right) \\
= : &~ \tilde I_1 + \tilde I_2 + \tilde I_3 +\tilde I_4.
\end{aligned}
\]
Now we compute $\tilde I_j$. First,
\[
\tilde I_1 =  -\frac{a}2 \int  \varphi' u_x^2  -\frac12  \int  \varphi' u^2  -\frac12 \int  \varphi' u^2 \eta - \frac12 \int \varphi u^2 \eta_x. 
\]
Second,
\[
\tilde I_3=  -\frac{c}2 \int  \varphi' \eta_x^2 -\frac12 \int  \varphi' \eta^2 + \frac12 \int  \varphi \eta_x u^2.
\]
Consequently,
\[
\tilde I_1+ \tilde I_3 = -\frac{a}2 \int  \varphi' u_x^2-\frac{c}2 \int  \varphi' \eta_x^2   -\frac12  \int  \varphi' (u^2 +\eta^2)  -\frac12 \int  \varphi' u^2 \eta.
\]
On the other hand,
\[
\begin{aligned}
\tilde I_2 =  &~  \int \varphi' u (1-\partial_x^2)^{-1}\left(a u_{xx} +u +u \eta \right) \\
= &~ {} a\int \varphi' u (1-\partial_x^2)^{-1} (u_{xx} -u) \\
&~ {}+ (1+ a)\int \varphi' u (1-\partial_x^2)^{-1} u   + \int \varphi' u (1-\partial_x^2)^{-1}\left(u \eta \right) \\
= &~ {}  - a\int \varphi' u^2 + (1+ a)\int \varphi' u (1-\partial_x^2)^{-1} u   + \int \varphi' u (1-\partial_x^2)^{-1}\left(u \eta \right).
\end{aligned}
\]
Similarly,
\[
\begin{aligned}
\tilde I_4 =  &~  \int \varphi' \eta (1-\partial_x^2)^{-1}\left(c\eta_{xx} +\eta +\frac12 u^2\right) \\
= &~ {} c \int \varphi' \eta (1-\partial_x^2)^{-1} (\eta_{xx} -\eta ) \\
&~ {}+ (1+ c)\int \varphi' \eta (1-\partial_x^2)^{-1} \eta   +\frac12 \int \varphi' \eta (1-\partial_x^2)^{-1}\left(u^2\right) \\
= &~ {}  - c\int \varphi' \eta^2 + (1+ c)\int \varphi' \eta (1-\partial_x^2)^{-1} \eta   + \frac12\int \varphi' \eta (1-\partial_x^2)^{-1}\left(u^2\right).
\end{aligned}
\]
We conclude that 
\[
\begin{aligned}
\frac{d}{dt} \mathcal I(t) = &~ \tilde I_1 + \tilde I_2 + \tilde I_3 +\tilde I_4 \\
=&~  {}  -\frac{a}2 \int  \varphi' u_x^2-\frac{c}2 \int  \varphi' \eta_x^2   -\frac12  \int  \varphi' (u^2 +\eta^2)  -\frac12 \int  \varphi' u^2 \eta \\
& ~ {} - a\int \varphi' u^2 + (1+ a)\int \varphi' u (1-\partial_x^2)^{-1} u   + \int \varphi' u (1-\partial_x^2)^{-1}\left(u \eta \right) \\
&~ {} - c\int \varphi' \eta^2 + (1+ c)\int \varphi' \eta (1-\partial_x^2)^{-1} \eta   + \frac12\int \varphi' \eta (1-\partial_x^2)^{-1}\left(u^2\right) \\
=&~  {}  -\frac{a}2 \int  \varphi' u_x^2-\frac{c}2 \int  \varphi' \eta_x^2   - \left(a+\frac12\right)   \int  \varphi' u^2 -\left( c+ \frac12 \right)  \int  \varphi' \eta^2 \\
& ~ {} + (1+ a)\int \varphi' u (1-\partial_x^2)^{-1} u  + (1+ c)\int \varphi' \eta (1-\partial_x^2)^{-1} \eta  \\
&~ {}     -\frac12 \int  \varphi' u^2 \eta   + \int \varphi' u (1-\partial_x^2)^{-1}\left(u \eta \right)  + \frac12\int \varphi' \eta (1-\partial_x^2)^{-1}\left(u^2\right) .
\end{aligned}
\]
This last equality proves \eqref{Virial0}. 
\end{proof}

\subsection{Virial functional $\mathcal J(t)$}

In what follows, we recall the system \eqref{boussinesq}-\eqref{Conds} written in the equivalent form
\begin{equation}\label{eq:abcd}
\begin{cases}
\pt \eta  = a \px u -(1+a)(1-\partial_x^2)^{-1}\px u - (1-\partial_x^2)^{-1}\px(u\eta) \\
\pt u  = c \px \eta -(1+c)(1-\partial_x^2)^{-1}\px \eta - (1-\partial_x^2)^{-1}\px(\frac12u^2).
\end{cases}
\end{equation}

\begin{lemma}\label{lem:J}
For any $t \in \R$,
\begin{equation}\label{eq:J-1}
\begin{aligned}
\frac{d}{dt} \mathcal J(t)=&~  (1+c)\int\vp'\eta^2 - c\int\vp' \eta_x^2  -(1+a)\int\vp' u^2 + a \int\vp' u_x^2\\
&-(1+c)\int\vp'\eta\nlop\eta + (1+a)\int\vp' u \nlop u\\
&+(1+a) \int \vp''u \nlop u_x + \frac{c}{2}\int\vp''' \eta^2\\
&-\frac12\int\vp' u^2\eta -\frac12\int\vp' \eta\nlop \left( u^2 \right)  \\
&+ \int \vp' u\nlop \left( u\eta \right) +\int \vp'' u\nlop ( u\eta )_x.
\end{aligned}
\end{equation}
\end{lemma}

\begin{proof}
We compute, using \eqref{eq:abcd},
\[\begin{aligned}
\frac{d}{dt} \mathcal J(t) =&\int\vp'\left(\eta_t u_x + \eta u_{xt}\right)\\
=&\int \vp'u_x\left(a u_x - (1+a)(1-\px^2)^{-1}u_x - (1-\px^2)^{-1}(u\eta)_x\right)\\
&+\int\vp \eta \left(c \eta_{xx} - (1+c)(1-\px^2)^{-1}\eta_{xx} - \frac12(1-\px^2)^{-1}( u^2 )_{xx}\right)\\
=: & ~J_1 +J_2.
\end{aligned}\]
We first deal with $J_1$. The integration by parts yields
\begin{equation}\label{eq:J1-1}
\begin{aligned}
J_1=& ~{} a \int \vp'u_x^2 - (1+a)\int\vp' u_x \nlop u_x - \int\vp' u_x \nlop (u\eta)_x\\
=& ~ {} a \int \vp'u_x^2+(1+a)\int\vp' u \nlop u_{xx}\\ 
&+ (1+a)\int\vp'' u \nlop  u_x \\
&+\int\vp' u \nlop (u\eta)_{xx} + \int\vp'' u \nlop (u\eta)_x.
\end{aligned}
\end{equation}
We use
\begin{equation}\label{eq:trick1}
\int \psi w \nlop z_{xx} = -\int \psi  wz + \int \psi  w \nlop z
\end{equation}
for the second and fourth terms in the right-hand side of \eqref{eq:J1-1} to obtain
\begin{equation}\label{eq:J1-2}
\begin{aligned}
J_1=&~{}-(1+a)\int \vp' u^2 + a \int \vp' u_x^2\\
&+ (1+a)\int\vp' u \nlop u +(1+a)\int\vp''u\nlop u_x\\
&-\int\vp'u^2\eta + \int\vp'u\nlop(u\eta) +\int \vp''u\nlop (u\eta)_x.
\end{aligned}
\end{equation}
For $J_2$, using the integration by parts and \eqref{eq:trick1} with the identity
$
ff_{xx} = \frac12(f^2)_{xx} - f_x^2,
$ 
 yields
\begin{equation}\label{eq:J2}
\begin{aligned}
J_2=&~ c\int \vp'\eta \eta_{xx} - (1+c)\int \vp' \eta \nlop \eta_{xx}\\
&-\frac12\int\vp'\eta \nlop (u^2)_{xx}\\
=&~ (1+c)\int\vp'\eta^2 - c\int\vp'\eta_x^2 + \frac{c}{2}\int\vp'''\eta^2\\
&-(1+c)\int\vp'\eta\nlop\eta\\
&+\frac12\int\vp'u^2\eta-\frac12\int\vp'\eta\nlop(u^2).
\end{aligned}
\end{equation}
By collecting all \eqref{eq:J1-2} and \eqref{eq:J2}, we have \eqref{eq:J-1}.
\end{proof}

\subsection{Virial functional $\mathcal K(t)$} 
In view of the $abcd$ system \eqref{eq:abcd}, one can realize that both equations have a sort of weakly symmetric structure (up to constants $a, c$, and nonlinearities). A slight modification in Lemma \ref{lem:J} (using the replacements $u \leftrightarrow \eta$, $a \leftrightarrow c$ and $u\eta \leftrightarrow \frac12u^2$) provides the following result for the virial functional $\mathcal{K}(t)$. The interested reader may check the details.

\begin{lemma}\label{lem:K}
For any $t \in \R$,
\begin{equation}\label{eq:K-1}
\begin{aligned}
\frac{d}{dt} \mathcal K(t)=&~ {} -(1+c)\int\vp'\eta^2 + c\int\vp'\eta_x^2 +(1+a)\int\vp' u^2 - a \int\vp' u_x^2\\
&+(1+c)\int\vp'\eta\nlop\eta - (1+a)\int\vp'u \nlop u\\
&+(1+c) \int \vp''\eta \nlop \eta_x + \frac{a}{2}\int\vp'''u^2\\
& - \int \vp'u\nlop(u\eta) +\frac12\int\vp'\eta\nlop(u^2)  \\
&+\frac12\int\vp'u^2\eta +\frac12\int \vp''\eta\nlop(u^2)_x.
\end{aligned}
\end{equation}
\end{lemma}

\medskip

\subsection{Modified virial} Now we construct a global virial from a linear combination of $\mathcal I(t),\mathcal  J(t)$ and $\mathcal K(t)$. 

\medskip

For $\alpha$ and $\beta$ real numbers, we define the modified virial
\be\label{H}
\mathcal H(t):= \mathcal H_{\al,\bt}(t):= \mathcal I(t) + \alpha \mathcal J(t) + \beta \mathcal K(t).
\ee
Thanks to Lemmas \ref{Virial_bous}, \ref{lem:J} and \ref{lem:K}, we obtain the following direct consequence:

\begin{proposition}[Decomposition of $\frac{d}{dt}\mathcal H(t)$]\label{prop:general virial}
Let $\eta$ and $u$ satisfy \eqref{eq:abcd}. For any $\alpha, \beta \in \R$ and any $t \in \R$, we have the decomposition
\begin{equation}\label{eq:gvirial}
\frac{d}{dt}\mathcal H(t) = \mathcal Q(t) + \mathcal{SQ}(t) + \mathcal{NQ}(t),
\end{equation}
where $\mathcal Q(t) =\mathcal Q[u,\eta](t) $ is the quadratic form
\begin{equation}\label{eq:leading}
\begin{aligned}
\mathcal Q(t) :=& ~\left((1+c)(\alpha-\beta-1) + \frac12\right)\int \vp' \eta^2 +c\Big(\beta - \alpha -\frac12\Big)\int \vp' \eta_x^2 \\
&+\left((1+a)(\beta -\alpha-1) + \frac12\right)\int \vp' u^2 +a\Big(\alpha-\beta-\frac12\Big)\int \vp' u_x^2 \\
&+(1+c)(\beta - \alpha + 1)\int \vp' \eta \nlop \eta\\
&+(1+a)(\alpha - \beta +1)\int \vp' u \nlop u ,
\end{aligned}
\end{equation}
$ \mathcal{SQ}(t)$ represents lower order quadratic terms not included in $\mathcal Q(t)$:
\begin{equation}\label{eq:small linear}
\begin{aligned}
 \mathcal{SQ}(t) :=& ~ \beta(1+c)\int\vp'' \eta \nlop \eta_x + \alpha(1+a)\int\vp'' u \nlop u_x\\
&+\frac{\alpha c}{2} \int \vp''' \eta^2 + \frac{\beta a}{2} \int \vp'''  u^2,
\end{aligned}
\end{equation}
and  $\mathcal{NQ}(t)$ are truly cubic order terms or higher:
\begin{equation}\label{eq:nonlinear}
\begin{aligned}
 \mathcal{NQ}(t) :=& ~\frac12(\beta - \alpha -1)\int\vp' u^2\eta + \frac12(\beta -\alpha + 1)\int \vp' \eta \nlop (u^2)\\
&+ (\alpha - \beta +1)\int \vp'  u \nlop (u\eta) + \frac{\beta}{2}\int \vp'' \eta \nlop (u^2)_x\\
&+ \alpha\int \vp'' u \nlop (u\eta)_x.
\end{aligned}
\end{equation}
\end{proposition}

\begin{remark}
Note that all members in $ \mathcal{SQ}(t)$ \eqref{eq:small linear} contains terms  of the form $\vp^{(n)}(x)$, $n=2,3$, and it makes each term be as small as we want by choosing an appropriate weight function and using rescaling arguments. Also, all terms in $ \mathcal{NQ}(t)$ will  not be of great importance, because they will be bounded by quantities smaller than any term in $\mathcal Q(t)$.
\end{remark}

\bigskip

\section{Estimates in canonical variables and positivity}\label{sec:general virial}

\medskip

\subsection{Passage to canonical variables}

We focus on the quadratic form $\mathcal{Q}(t)$ in \eqref{eq:leading}. We introduce canonical variables for $u$ and $\eta$ (see Definition \ref{Can_Var}) as follows:
\begin{equation}\label{eq:fg}
f := \nlop u \quad \mbox{and} \quad g := \nlop \eta.
\end{equation} 
Note that, for $u, \eta \in H^1$, one has $f$ and $g$ in $H^3$. A direct calculation shows the following key relationships between $f$ and $u$ (resp. $g$ and $\eta$), see also Lemmas \ref{lem:L2 comparable} and \ref{lem:H1 comparable} for similar statements.
\begin{lemma}
One has
\begin{equation}\label{eq:L2}
\int \vp' u^2 = \int\vp'\left(f^2 + 2f_x^2 + f_{xx}^2\right) - \int \vp'''f^2,
\end{equation}
\begin{equation}\label{eq:H1}
\int \vp'u_x^2 = \int\vp'\left(f_x^2 + 2f_{xx}^2 + f_{xxx}^2\right) - \int \vp'''f_x^2
\end{equation}
and
\begin{equation}\label{eq:nonlocal}
\int \vp' u \nlop u = \int\vp'\left(f^2 + f_x^2\right) - \frac12\int \vp'''f^2.
\end{equation}
\end{lemma}

\medskip

Using \eqref{eq:L2}-\eqref{eq:nonlocal}, we can rewrite the quadratic form $\mathcal{Q}(t)$ as follows:
\begin{lemma}\label{lem:leading}
Let $f$ and $g$ be canonical variables of $u$ and $\eta$ as in \eqref{eq:fg}. Consider the quadratic form $\mathcal Q(t)$ given in \eqref{eq:leading}. Then we have
\begin{equation}\label{eq:leading-1}
\begin{aligned}
\mathcal Q(t) =& \int \vp' \Big( A_1 f^2 + A_2 f_x^2 + A_3 f_{xx}^2 + A_4 f_{xxx}^2\Big)\\
&+\int \vp' \Big( B_1 g^2 + B_2 g_x^2 + B_3 g_{xx}^2 + B_4 g_{xxx}^2\Big)\\
&+\int \vp''' \Big(D_{11}f^2 + D_{12}f_x^2 + D_{21}g^2 + D_{22}g_x^2\Big),
\end{aligned}
\end{equation}
where
\begin{equation}\label{eq:A1B1}
A_1 = B_1 = \frac12>0,
\end{equation}

\begin{equation}\label{eq:A2B2}
A_2 = \beta - \alpha -\frac{3a}{2}, \qquad B_2 = \alpha - \beta - \frac{3c}{2},
\end{equation}

\begin{equation}\label{eq:A3B3}
A_3 = (1-a)(\beta - \alpha) -2a - \frac12, \qquad B_3 = (1-c)(\alpha- \beta) -2c - \frac12,
\end{equation}

\begin{equation}\label{eq:A4B4}
A_4= a\left(\alpha - \beta - \frac12\right), \qquad B_4= c\left(\beta - \alpha - \frac12\right),
\end{equation}
and
\begin{equation}\label{eq:D1D2}
\begin{aligned}
&D_{11} =-\frac12(1+a)(\beta - \alpha - 1) - \frac12 , \qquad D_{12}=-a\left(\alpha - \beta - \frac12\right),\\ 
&D_{21}= -\frac12(1+c)(\alpha - \beta - 1) - \frac12,\qquad D_{22}=-c\left(\beta - \alpha - \frac12\right).
\end{aligned}
\end{equation}
\end{lemma}

\begin{proof}
Direct from the substitution of \eqref{eq:L2}-\eqref{eq:nonlocal} into \eqref{eq:leading}.
\end{proof} 

\begin{remark}\label{rem:small linear}
A slight modification of \eqref{eq:nonlocal} also allows us to rewrite the  first two terms in $\mathcal{SQ}(t)$ \eqref{eq:small linear} as follows: 
\begin{equation}\label{eq:small linear1} 
\beta(1+c)\int\vp''\eta \nlop \eta_x = - \frac12\beta(1+c) \int \vp''' (g^2 - g_x^2),
\end{equation}
and
\begin{equation}\label{eq:small linear2}
\alpha(1+a)\int\vp''u \nlop u_x = - \frac12 \alpha(1+a) \int \vp''' (f^2 - f_x^2).
\end{equation}
These terms, however, will be negligible compared with the terms in $\mathcal{Q}(t)$. 
\end{remark}

\subsection{Positivity}\label{Positivity} We want to study positivity (or negativity) properties of the bilinear form $\mathcal{Q}(t)$, namely, to decide under which conditions $\mathcal{Q}(t)$ in \eqref{eq:leading-1} has a unique definite sign.

\medskip

It turns out that $A_1$ and $B_1$ in \eqref{eq:A1B1} are positive, so the only possibility is to obtain $\mathcal{Q}(t)$ positive definite. Then, a sufficient condition for getting this goal is to impose
\be\label{positive}
A_k>0, \quad B_k>0, \quad k=2,3,4.
\ee
Note that the coefficients $D_k$ may have any possible sign, this is because we will assume that $\varphi'''$ is as smaller as we want, compared to $\varphi'$. 

\medskip

The following lemma characterizes condition \eqref{positive} in terms of the  \emph{dispersion-like parameters} introduced in Definition \ref{Dis_Par}.
\begin{lemma}[Positivity vs. dispersion]\label{lem:conditions}
Let $a, c < 0$. Then,  $A_k > 0$ and $B_k > 0$ for $k=2,3,4$  if and only if $(a,c)$ are dispersion-like parameters. That is to say,
\begin{equation}\label{eq:conditions1}
3(a+c) + 2 - 8ac < 0.
\end{equation} 
\end{lemma}

Later we will prove (see Lemma \ref{b29}) that \eqref{eq:conditions1} is satisfied provided $b>\frac29$. See Fig. \ref{Fig:5} for a picture of the region of pairs $(a,c)\in\R^2$ for which \eqref{eq:conditions1} is satisfied, in terms of the parameter $b>\frac29$.

\begin{proof}[Proof of Lemma \ref{lem:conditions}]
First of all, we have from \eqref{eq:A4B4} and the (negative) signs of $a,c$ that $A_4$ and $B_4$ are positive whenever $(\alpha,\beta)$ belongs to the set
\begin{equation}\label{eq:conditions2}
\mathcal{A}_4 := \left\{(\alpha,\beta) \in \R^2: \alpha - \frac12 < \beta < \alpha + \frac12 \right\},
\end{equation}
This is a set consisting of a diagonal band around zero, see Fig. \ref{Fig:1}. Note that $\mathcal{A}_4 $ is never empty, and $(0,0)$ is always included (recall that $(0,0)$ corresponds to the case where $\mathcal J$ and $\mathcal K$ are not included (nor necessary) in $\mathcal H$, see \eqref{H}).

\begin{figure}[h!]
\begin{center}
\begin{tikzpicture}[scale=0.8]
\filldraw[thick, color=lightgray!25] (-1,-1)--(1,-1) -- (5,3) --(5,5) -- (-1,-1);
\draw[thick,dashed] (-1,-1)--(5,5);
\draw[thick,dashed] (1,-1)--(5,3);
\draw[->] (-1,1) -- (5,1) node[below] {$\alpha$};
\draw[->] (2,-1) -- (2,5) node[right] {$\beta$};
\node at (2.4,-0.1){$-\frac 12$};
\node at (1.7,2.2){$\frac 12$};
\node at (4.3,3.3){$\mathcal{A}_4$};
\end{tikzpicture}
\qquad 
\begin{tikzpicture}[scale=0.8]
\filldraw[thick, color=lightgray!25] (-1,-1)--(1,-1) -- (5,3) --(5,5) -- (-1,-1);
\draw[thick,dashed] (-1,-1)--(5,5);
\draw[thick,dashed] (1,-1)--(5,3);
\draw[->] (-1,1) -- (5,1) node[below] {$\alpha$};
\draw[->] (2,-1) -- (2,5) node[right] {$\beta$};
\node at (2.3,-0.2){$\frac {3a}2$};
\node at (1.5,2.2){$-\frac {3c}2$};
\node at (4.3,3.3){$\mathcal{A}_2$};
\end{tikzpicture}
\end{center}
\caption{(\emph{Left}). The set $\mathcal{A}_4 $ defined in \eqref{eq:conditions2}. Note that $(0,0)$ is always included in  $\mathcal{A}_4 $, meaning that, for having $A_4$ and $B_4$ both positive, $\mathcal J$ and $\mathcal K$ are not needed. (\emph{Right}). The set $\mathcal{A}_2 $ defined in \eqref{eq:conditions3}. This set may or may not be contained in $\mathcal{A}_4$, depending on the values of $a$ and $c$. It also may be the case that $\mathcal{A}_2 $ contains $\mathcal{A}_4 $. Its boundary never crosses the origin, and it is not necessarily symmetric with respect to the line $\bt = \al$. Therefore, for having $A_2,B_2>0$, both $\mathcal J$ and $\mathcal K$ are not needed.} \label{Fig:1}
\end{figure}
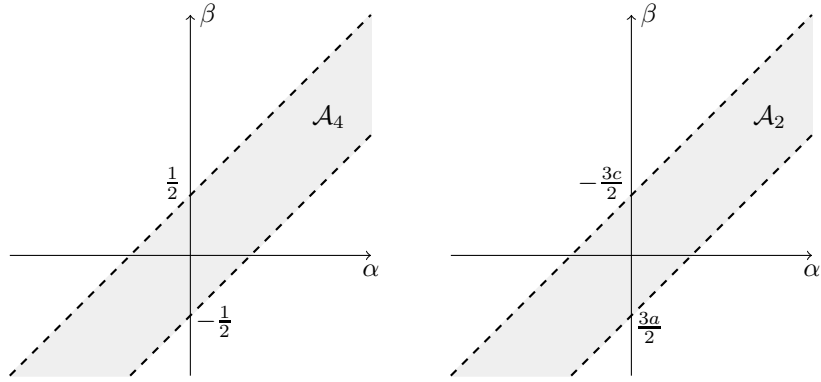
 
\medskip

From \eqref{eq:A2B2}, we see that both $A_2 $ and $B_2 $ are positive if $(\alpha,\beta)$ belongs to the set
\begin{equation}\label{eq:conditions3}
\mathcal{A}_2 := \left\{(\alpha,\beta) \in \R^2: \alpha + \frac{3a}{2} < \beta < \alpha - \frac{3c}{2}\right\}.
\end{equation}
(See Fig. \ref{Fig:1}.) This set makes sense if and only if 
\[
\frac{3a}{2} < - \frac{3c}{2} \Longleftrightarrow a + c < 0,
\] 
which is a direct consequence of the fact that $a<0$ and $c<0$. Hence, always the set $\mathcal{A}_2$ is nonempty and $(0,0)$ is contained. Note additionally that $ \frac{3a}{2}<0$ and $- \frac{3c}{2}>0$, so the boundary of $\mathcal{A}_2$ never crosses the origin. However, depending on the values of $a$ and $c$, they might contain, be contained, or not contain $\mathcal{A}_4$.

\medskip

Finally, we deal with the positivity of the terms $A_3$ and $B_3$ in \eqref{eq:A3B3}. It is not difficult to show that  $A_3 > 0 $ and $B_3 > 0$ whenever $(\alpha,\beta)$ belongs to the set
\begin{equation}\label{eq:conditions4}
\mathcal{A}_3 := \left\{(\alpha,\beta) \in \R^2: \alpha + \frac{1+4a}{2(1-a)} < \beta < \alpha - \frac{1+4c}{2(1-c)} \right\}.
\end{equation}
(See Fig. \ref{Fig:3}.) Note that the set $\mathcal{A}_3$ is well-defined, since $a,c < 0$. We know that $\mathcal{A}_3 \neq \emptyset$ if and only if 
\[
\frac{1+4a}{2(1-a)} < - \frac{1+4c}{2(1-c)},
\] 
which is nothing but the dispersion-like condition \eqref{eq:conditions1}. Therefore, we have $\mathcal{A}_2$, $\mathcal{A}_3$ and $\mathcal{A}_4$ nonempty.

\begin{figure}[h!]
\begin{center}
\begin{tikzpicture}[scale=0.8]
\filldraw[thick, color=lightgray!25] (-1,-1)--(1,-1) -- (5,3) --(5,5) -- (-1,-1);
\draw[thick,dashed] (-1,-1)--(5,5);
\draw[thick,dashed] (1,-1)--(5,3);
\draw[->] (-1,1) -- (5,1) node[below] {$\alpha$};
\draw[->] (2,-1) -- (2,5) node[right] {$\beta$};
\node at (2.7,-0.1){$ \frac{1+4a}{2(1-a)}$};
\node at (1.2,2.2){$- \frac{1+4c}{2(1-c)}$};
\node at (4.3,3.3){$\mathcal{A}_3$};
\end{tikzpicture}
\end{center}
\caption{The set $\mathcal{A}_3 $ defined in \eqref{eq:conditions4}. Contrary to other sets $\mathcal{A}_4$ and $\mathcal{A}_2$, this set \emph{may not contain} the point $(0,0)$, depending on the sign of one of the terms $ \frac{1+4a}{2(1-a)}$ or $- \frac{1+4c}{2(1-c)}$. In this case, the perturbations $\mathcal J$ and $\mathcal K$ in \eqref{H} are fundamental and necessary.} \label{Fig:3}
\end{figure}
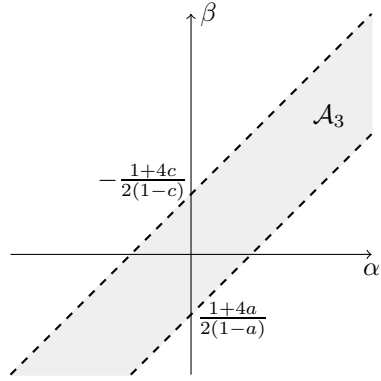
 
\medskip

It remains to show $\mathcal{A}_2 \cap \mathcal{A}_3 \cap \mathcal{A}_4 \neq \emptyset$, namely there exists parameters $(\al,\bt)$ such that \eqref{positive} is satisfied.

\medskip

First of all, the conditions $a, c < 0$ guarantee
\[\frac{3a}{2}  < \frac12 \qquad \mbox{and} \qquad -\frac12 < - \frac{3c}{2},\]
which imply $\mathcal{A}_4 \cap \mathcal{A}_2 \neq \emptyset$ without any further conditions on $a$ and $c$. On the other hand, $\mathcal{A}_4 \cap \mathcal{A}_3 \neq \emptyset$ whenever
\begin{equation}\label{eq:conditions5}
\frac{1+4a}{2(1-a)} < \frac12 \qquad \mbox{and} \qquad -\frac12  < - \frac{1+4c}{2(1-c)}.
\end{equation}
The condition $a, c <0$ immediately implies \eqref{eq:conditions5}, and hence $\mathcal{A}_4 \cap \mathcal{A}_3 \neq \emptyset$ without any further conditions on $a$ and $c$.

\medskip

Now we deal with the proof of $\mathcal{A}_2 \cap \mathcal{A}_3 \cap \mathcal{A}_4 \neq \emptyset$. Note that each set $\mathcal{A}_i$, $i=2,3,4$, describes an unbounded strip on $(\alpha,\beta)$-plane with ``the same slope''. Hence, it suffices to show that
\begin{equation}\label{eq:intersection}
\mathcal{A}_2 \cap \mathcal{A}_3 \neq \emptyset \; \wedge\;  \mathcal{A}_2 \cap \mathcal{A}_4 \neq \emptyset \; \wedge \; \mathcal{A}_3 \cap \mathcal{A}_4 \neq \emptyset.
\end{equation}
Indeed, if $\mathcal{A}_2 \cap \mathcal{A}_3 = \mathcal{A}_2$ (or $\mathcal{A}_3$), \eqref{eq:intersection} immediately implies $\mathcal{A}_2 \cap \mathcal{A}_3 \cap \mathcal{A}_4 \neq \emptyset$. Hence, suppose that $\mathcal{A}_2 \nsubseteq \mathcal{A}_3$ and $\mathcal{A}_3 \nsubseteq \mathcal{A}_2$. In this case $\mathcal{A}_2 \cap \mathcal{A}_3$ describes a new band constructed by upper and lower edges from $\mathcal{A}_2$ and $\mathcal{A}_3$, respectively, (or vice versa). Without loss of generality, we assume the new band is constructed by upper edge from $\mathcal{A}_2$ and lower edge from $\mathcal{A}_3$. On the other hand, $\mathcal{A}_2 \cap \mathcal{A}_4 \neq \emptyset$ and $\mathcal{A}_3 \cap \mathcal{A}_4 \neq \emptyset$ (and the open character of each involved set) require that the upper edge of $\mathcal{A}_4$ should be positioned above the lower edge of $\mathcal{A}_3$, and the lower edge of $\mathcal{A}_4$ should be positioned below the upper edge of $\mathcal{A}_2$, respectively. This exactly implies $(\mathcal{A}_2 \cap \mathcal{A}_3 )\cap \mathcal{A}_4 \neq \emptyset$. This is the only possible case, because all $\mathcal A_j$ describe bands with the same slope. Therefore, we only need to check \eqref{eq:intersection}. 

\medskip

Let us prove the remaining part of \eqref{eq:intersection}, namely $\mathcal{A}_2 \cap \mathcal{A}_3 \neq \emptyset$. This happens if 
\begin{equation}\label{eq:conditions6}
\frac{1+4a}{2(1-a)} < -\frac{3c}{2} \qquad \mbox{and} \qquad \frac{3a}{2}  < - \frac{1+4c}{2(1-c)}.
\end{equation}
A direct calculation under $a,c<0$ yields that \eqref{eq:conditions6} is equivalent to the conditions
\be\label{Final_condition}
3ac > 1 + 4a + 3c, \qquad  3ac > 1 + 3a + 4c.
\ee
On the $(a,c)$-plane, we can see that the region of $(a,c)$ satisfying \eqref{Final_condition} and $a,c<0$ covers the region of $(a,c)$ satisfying \eqref{eq:conditions1}. Indeed, consider the following two equations
\begin{equation}\label{eq:region1}
 8ac -2 -3(a+c)= 0 \quad \mbox{and} \quad 3ac - 1 - 4a - 3c = 0.
\end{equation}
Note that if $a < 0$, we know
\begin{equation}\label{eq:region2}
8a - 3 \neq 0 ~(< 0), \quad \mbox{and} \quad 3a-3 \neq0 ~(<0).
\end{equation}
For $a<0$, \eqref{eq:region1} can be expressed as implicit function formulas of $c$ in terms of $a$ as follows:
\[c=\Gamma_1(a) := \frac{3a+2}{8a-3} \quad \mbox{and} \quad c=\Gamma_2(a):=\frac{4a+1}{3a-3}.\]
It is known that if $\Gamma_1(a) < \Gamma_2(a)$, the region of $c < \Gamma_2(a)$ covers the region of $c < \Gamma_1(a)$ on $(a,c)$-plane, and hence the region of $(a,c)$ satisfying \eqref{eq:conditions1} is contained to the region of $(a,c)$ satisfying $3ac - 1 - 4a - 3c > 0$ for $a < 0$ due to \eqref{eq:region2}. A straightforward calculation indeed gives
\begin{equation}\label{eq:region}
\frac{3a+2}{8a-3} < \frac{4a+1}{3a-3}\; \Longleftrightarrow  \;23a^2 -a +3 >0,
\end{equation}
for all $a<0$. A similar argument holds for $3ac > 1 + 3a + 4c$.

\medskip

Thus, if $a < 0$ and $c < 0$ satisfy \eqref{eq:conditions1}, we can always choose $(\alpha,\beta) \in \mathcal{A}_2 \cap \mathcal{A}_3 \cap \mathcal{A}_4$ such that $A_k >0$ and $B_k>0$, $k=1,2,3,4$. The proof is complete.
\end{proof}

In what follows we study whether or not condition  \eqref{eq:conditions1} can be satisfied for different values of $b>\frac16$. 

\medskip

Going back to the original parameters $a$ and $b$ from \eqref{Conds0} and \eqref{Conds2} (recall that $a$ and $c$ were changed to $a/b$ and $c/b$ respectively), \eqref{eq:conditions1} can be rewritten as the fact that the set
\[
\mathcal{B}_4(b) := \left\{(a,c) \in \R^2  ~ : ~ 3b(a+c) +2b^2 -8ac < 0\right\},
\]
is nonempty. 

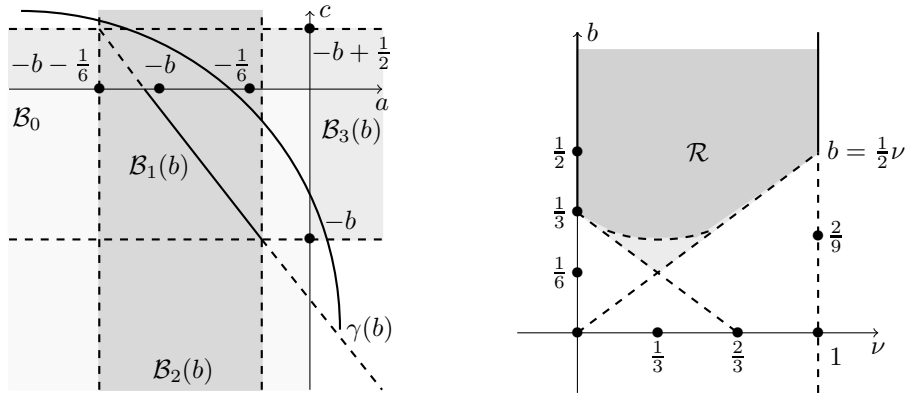
\begin{figure}[h!]
\begin{center}
\begin{tikzpicture}[scale=0.8]
\filldraw[thick, color=lightgray!30] (-1,1.5)--(5.2,1.5) -- (5.2,5) --(-1,5) -- (-1,1.5);
\filldraw[thick, color=lightgray!10] (-1,-1)--(-1,4) -- (4,4) --(4,-1) -- (-1,-1);
\filldraw[thick, color=lightgray!60] (0.5,-1)--(0.5,5.3) -- (3.2,5.3) --(3.2,-1) -- (0.5,-1);
\draw[thick, color=black] (1.25,4) -- (3.2,1.5);
\draw[thick,dashed] (0.5,5) -- (1.25,4);
\draw[thick,dashed] (3.2,1.5) -- (5.2,-1);
\draw[thick,dashed] (0.5,-1)--(0.5,5.3);
\draw[thick,dashed] (3.2,-1)--(3.2,5.3);
\draw[thick,dashed] (-1,1.5)--(5.2,1.5);
\draw[thick,dashed] (-1,5)--(5.2,5);
\draw[->] (-1,4) -- (5.2,4) node[below] {$a$};
\draw[->] (4,-1) -- (4,5.3) node[right] {$c$};
\node at (1.9,-0.7){$ \mathcal{B}_2(b)$};
\node at (3,4){$\bullet$};
\node at (2.7,4.4){$-\frac16$};
\node at (1.5,4){$\bullet$};
\node at (1.5,4.4){$-b$};
\node at (0.5,4){$\bullet$};
\node at (-0.3,4.4){$-b -\frac16$};
\node at (4,1.5){$\bullet$};
\node at (4.5,1.8){$-b$};
\node at (4,5){$\bullet$};
\node at (4.7,4.6){$-b+\frac12$};
\node at (4.7,3.3){$ \mathcal{B}_3(b)$};
\node at (1.5,2.7){$ \mathcal{B}_1(b)$};
\draw[thick] (4.5,0) arc (0:90:5.3);
\node at (5,0){$\gamma(b)$};
\node at (-0.7,3.5){$\mathcal{B}_0$};
\end{tikzpicture}
\qquad 
\qquad 
\begin{tikzpicture}[scale=0.8]
\filldraw[thick, color=lightgray!30] (0,4.7)--(0,2) -- (4/3,1) --(4,3) -- (4,4.7) -- (0,4.7);
\filldraw[thick, color=lightgray!70] (0,4.7)--(0,2) -- (0.6,1.7) --(1,1.6) --(4/3,1.6) --(1.97,1.6) --(4,3) -- (4,4.7) -- (0,4.7);
\draw[thick,dashed] (4,-1) -- (4,3);
\draw[thick] (4,3) -- (4,5);
\draw[thick,dashed] (0,0) -- (4,3);
\draw[thick,dashed] (0,2)--(8/3,0);
\draw[thick] (0,2) -- (0,5);
\draw[->] (-1,0) -- (5,0) node[below] {$\nu$};
\draw[->] (0,-1) -- (0,5) node[right] {$b$};
\node at (0,0){$\bullet$};
\node at (4,0){$\bullet$};
\node at (4.3,-0.4){$1$};
\node at (4/3,0){$\bullet$};
\node at (8/3,-0.4){$\frac23$};
\node at (4/3,-0.4){$\frac13$};
\node at (8/3,0){$\bullet$};
\node at (0,2){$\bullet$};
\node at (-0.3,2){$\frac 13$};
\node at (0,1){$\bullet$};
\node at (-0.3,1){$\frac16$};
\node at (0,3){$\bullet$};
\node at (4,1.6){$\bullet$};
\node at (4.3,1.6){$\frac29$};
\node at (-0.3,3){$\frac12$};
\node at (4.8,3){$b=\frac12\nu$};
\draw[thick,dashed] (0.5,1.7) arc (250:290:2.5);
\node at (2,3){$ \mathcal R$};
\end{tikzpicture}
\end{center}
\caption{(\emph{Left.}) Schematic representation of the action of the set $\mathcal{B}_4(b)$ in Fig. \ref{Fig:0}, in the case where $b>\frac29\sim 0.22 >\frac16\sim 0.17$. The region \emph{below} the continuous curve $\gamma(b)$ given by the equation $c= -b\left(\frac{2b+3a}{3b-8a} \right)$ represents the admissible $(a,c)$ given in  $\mathcal{B}_4(b)$. Note that $(a,c)$ cannot be arbitrarily small. (\emph{Right.}) The dark region $\mathcal R$ (subset of $\mathcal R_0$ in Fig. \ref{Fig:0}, right panel) represents all the pairs $(\nu,b)$ for which $\mathcal A_3$ is nonempty, and therefore Theorem \ref{Thm1} is valid. Note that $b=\frac29$ represents the bottom of this set, part of an hyperbola, see Appendix \ref{B} for more details. Probably Theorem \ref{Thm1} is still valid for a small portion of the remaining admissible region (light shadowed) below $b=\frac29$, but a more involved proof is needed. Finally note that at $(\nu,b)=(\frac13,\frac29)$, one has $(a,c)=(-\frac{1}{18},-\frac{1}{18}).$} \label{Fig:5}
\end{figure}

\medskip

Recall the sets $\mathcal{B}_0, \mathcal{B}_1(b), \mathcal{B}_2(b)$ and $ \mathcal{B}_3(b)$ defined in \eqref{Bj}, and Fig. \ref{Fig:0} and Lemma \ref{1o6}, which characterize the values of $b$ for which the $abcd$ system makes sense. In the following computations, we assume $b>\frac16$. 

\medskip

Let $\mathcal{P}_b := \mathcal{B}_0 \cap \mathcal{B}_1(b) \cap \mathcal{B}_2(b) \cap \mathcal{B}_3(b) \cap \mathcal{B}_4(b)$. As we seen in Subsection \ref{abc Conds},  $\mathcal{B}_0 \cap \mathcal{B}_1(b) \cap \mathcal{B}_2(b) \cap \mathcal{B}_3(b)$ represents a diagonal left on $(a,c)$-plane. Moreover, $\mathcal{B}_4(b)$ describes the region below the continuous curve $\gamma(b)$ given by the equation
\[
c= -b\left(\frac{2b+3a}{3b-8a} \right), \quad a\neq \frac38 b>0,
\]
see Fig. \ref{Fig:5} (recall that $a,c<0$ and that $c= -\frac 23b$ at $a=0$ is always above the minimal value $c=-b$ of $\mathcal{B}_1(b)$). Hence, a sufficient condition for having $\mathcal{P}_b \neq \emptyset$ is that the curve $\gamma(b)$  intersects the segment $ \mathcal{B}_1(b)$ at two different points, which is equivalent to the following equations having two roots for $(a,c)$
\begin{equation}\label{eq:solve}
\begin{cases}
a+c = \frac13 - 2b\\
c= -b\left(\frac{2b+3a}{3b-8a} \right).
\end{cases}
\end{equation} 
Solving \eqref{eq:solve} is reduced to $3456 b^2 - 1056b +64 >0$, which gives $b < \frac1{12} ~(<\frac16)$ or $b > \frac 29~ (>\frac16)$. Thus, we conclude that $\mathcal{P}_b \neq \emptyset$ if $b > \frac 29$. So far, we have proved 

\begin{lemma}[Sufficient condition for having dispersion-like parameters]\label{b29}
We have $\mathcal{P}_b =\mathcal{B}_0 \cap \mathcal{B}_1(b) \cap \mathcal{B}_2(b) \cap \mathcal{B}_3(b) \cap \mathcal{B}_4(b) \neq \emptyset$ provided $b>\frac29.$ Moreover, this set represents a segment of line in the $(a,c)$ variables.
\end{lemma}

This result proves the second item in Lemma \ref{Sufficient}.

\medskip

In Appendix \ref{B} we will prove that $\mathcal{P}_b$ follows a precise parametric characterization. Finally, let us define 
\begin{equation}\label{eq:parameter set}
\mathcal{P} :=\mathcal{P}((a,b,c)) = \bigcup_{b > \frac29} \mathcal{P}_b, \qquad \mathcal R:= \left\{  (\nu,b) \in \mathcal{R}_0 ~ :~  (a,b,c) \in \mathcal P_b \neq \emptyset \right\}.
\end{equation}
This set $\mathcal{P}$ describes the parameter condition of $(a,b,c)$, which allows both well-defined model in physical sense and dispersive phenomena in mathematical sense. 

\medskip

So far, we have proved the following result:

\begin{lemma}[Positivity of the quadratic form $\mathcal Q(t)$]\label{Positivity_final}
Let $a,c<0$ satisfying \eqref{Conds}. Under the dispersion-like condition \eqref{eq:conditions1} on $(a,b,c)$, there exist parameters $(\al,\bt) \in \R^2$ (not necessarily different from (0,0)), such that $\mathcal H(t) = \mathcal I(t) + \al \mathcal J(t)+\bt \mathcal K(t)$ obeys the following decomposition for its time derivative:
\be\label{Positivity_final_1}
\frac{d}{dt}\mathcal H(t) = \mathcal Q(t) + \mathcal{SQ}(t) + \mathcal{NQ}(t),
\ee
where
\begin{enumerate}
\item $\mathcal Q(t)$ is given by the expression
\be\label{eq:leading-1a}
\begin{aligned}
\mathcal Q(t) =& \int \vp' \Big( A_1 f^2 + A_2 f_x^2 + A_3 f_{xx}^2 + A_4 f_{xxx}^2\Big)\\
&+\int \vp' \Big( B_1 g^2 + B_2 g_x^2 + B_3 g_{xx}^2 + B_4 g_{xxx}^2\Big)\\
&+\int \vp''' \Big(D_{11}f^2 + D_{12}f_x^2 + D_{21}g^2 + D_{22}g_x^2\Big),
\end{aligned}
\ee
for $f= (1-\px^2)^{-1} u$ and $g= (1-\px^2)^{-1} \eta$.
\smallskip
\item {\bf Positivity.} We have $A_i,B_j>0$, and $D_{ij}\in \R$.
\smallskip
\item $\mathcal{SQ}(t)$, $\mathcal{NQ}(t)$ are given in \eqref{eq:small linear}-\eqref{eq:nonlinear}.
\end{enumerate}
\end{lemma}

\bigskip

\section{Integrability in time}\label{5}

\medskip
The purpose of this Section is to obtain an integrability-in-time estimate for  the terms appearing in Lemma \ref{Positivity_final}. This amounts to estimate the various terms in \eqref{Positivity_final_1} using the stability of the zero solution.

\begin{proposition}[Decay in compact intervals]\label{prop:virial}
Let $(a,b,c)$ be \emph{dispersion-like} parameters defined as in Definition \ref{Dis_Par}. Let $(u,\eta)(t)$ be $H^1 \times H^1$ global solutions to \eqref{boussinesq} such that \eqref{Smallness} holds. Then, we have for large (but fixed) $\lambda \gg 1$ that
\begin{equation}\label{eq:virial1}
\int_2^{\infty} \int \sech^2 \left(\frac{x}{\lambda}\right) \left(u^2 + (\px u)^2 + \eta^2 + (\px \eta)^2 \right)(t,x) \, dx\,dt \lesssim \lambda\ve^2.
\end{equation}
As an immediate consequence, there exists an increasing sequence of time $\{t_n\}$ $(t_n \to \infty$ as $n \to \infty)$ such that
\begin{equation}\label{eq:virial2}
\int \sech^2 \left(\frac{x}{\lambda}\right) \left(u^2 + (\px u)^2 + \eta^2 + (\px \eta)^2 \right)(t_n,x) \; dx \longrightarrow 0 \mbox{ as } n \to \infty. 
\end{equation}
\end{proposition}

\begin{remark}
Note that estimate \eqref{eq:virial1} completes the proof of \eqref{eq:virial1_intro}.
\end{remark}

\begin{proof}[Proof of Proposition \ref{prop:virial}]
Let us choose $\vp(x) = \lambda\tanh\big(\frac{x}{\lambda}\big)$ in Lemma \ref{Positivity_final}, for $\lambda >1$ to be chosen later. Note that
\begin{equation}\label{eq:sech}
\begin{aligned}
\vp' (x)= &~ {} \sech^2\Big(\frac{x}{\lambda} \Big), \quad |\vp'' (x)|\leq  \frac{2}{\la} \sech^2\Big(\frac{x}{\lambda} \Big) = \frac2{\la} \vp'(x) ,\\
 \mbox{and} & \quad |\vp'''(x)| \le \frac{1}{\lambda^2} \sech^2\Big(\frac{x}{\lambda} \Big) = \frac1{\la^2} \vp'(x).
\end{aligned}
\end{equation}
From Lemma \ref{Positivity_final}, choosing a sufficiently large $\lambda > 1$, the rest terms in \eqref{eq:leading-1a}, characterized by the weight $\vp'''$, cannot change the sign of $A_k$ and $B_k$, $k=1,2,3,4$, due to \eqref{eq:sech}. Therefore, they can be easily absorbed in \eqref{eq:leading-1a}:
\[
\begin{aligned}
\mathcal Q(t) \geq & ~{} \frac34 \int \vp' \Big( A_1 f^2 + A_2 f_x^2 + A_3 f_{xx}^2 + A_4 f_{xxx}^2\Big)\\
& ~{}+ \frac34 \int \vp' \Big( B_1 g^2 + B_2 g_x^2 + B_3 g_{xx}^2 + B_4 g_{xxx}^2\Big).
\end{aligned}
\]
Similarly,  $\mathcal{SQ}(t)$ in  \eqref{eq:small linear} can be characterized in terms of canonical variables using e.g. \eqref{eq:small linear1} and \eqref{eq:small linear2}. Once again, because of the weight terms $\vp''$ and $\vp'''$ and \eqref{eq:sech}, these terms are negligible with respect to the leading terms in \eqref{eq:leading-1a}. We conclude from \eqref{Positivity_final_1},
\[
\begin{aligned}
\frac{d}{dt}\mathcal H(t) \geq & ~{} \frac12 \int \vp' \Big( A_1 f^2 + A_2 f_x^2 + A_3 f_{xx}^2 + A_4 f_{xxx}^2\Big)\\
& ~{}+ \frac12 \int \vp' \Big( B_1 g^2 + B_2 g_x^2 + B_3 g_{xx}^2 + B_4 g_{xxx}^2\Big) + \mathcal{NQ}(t).
\end{aligned}
\]

Hence, Lemmas \ref{lem:L2 comparable} and \ref{lem:H1 comparable} give 
\[\begin{aligned}
\frac{d}{dt} \mathcal H(t) \ge&~{} C_1\int \sech^2\Big( \frac{x}{\lambda}\Big) u^2 + C_2\int \sech^2\Big( \frac{x}{\lambda}\Big) (\px u)^2 \\
&+C_3\int \sech^2\Big( \frac{x}{\lambda}\Big) \eta^2 + C_4\int \sech^2\Big( \frac{x}{\lambda}\Big) (\px \eta)^2\\
&+ \mathcal{NQ}(t),
\end{aligned}\]
for large $\lambda \gg 1$ and some $C_j >0$, $j=1,2,3,4$.

\medskip

On the other hand, Lemma \ref{lem:nonlinear1} easily gives
\[
\mathcal{NQ}(t) \lesssim (\norm{u}_{H^1} + \norm{\eta}_{H^1}) \int \vp' (u^2 + (\px u)^2 + \eta^2 + (\px \eta)^2).
\]
This estimate, together with the smallness condition of $u$ and $\eta$, it implies 
\begin{equation}\label{eq:virial1-1}
\begin{aligned}
\frac{d}{dt} \mathcal H(t)  \ge&~{} \wt{C}_1\int \sech^2\Big( \frac{x}{\lambda}\Big) u^2 + \wt{C}_2\int \sech^2\Big( \frac{x}{\lambda}\Big) (\px u)^2 \\
&+\wt{C}_3\int \sech^2\Big( \frac{x}{\lambda}\Big) \eta^2 + \wt{C}_4\int \sech^2\Big( \frac{x}{\lambda}\Big) (\px \eta)^2,
\end{aligned}
\end{equation}
for some $\wt{C}_j >0$, $j=1,2,3,4$. Integrating the both sides of \eqref{eq:virial1-1} in terms of $t$ on $[2,\infty)$ with $\norm{u}_{H^1} + \norm{\eta}_{H^1} \lesssim \ve$ exactly proves \eqref{eq:virial1}:
\[\int_2^{\infty} \int \sech^2 \left(\frac{x}{\lambda}\right) \left(u^2 + (\px u)^2 + \eta^2 + (\px \eta)^2\right)(t,x) \; dx dt \lesssim \lambda\ve^2.\]
The standard argument guarantees that there exists a sequence of time $t_n \to \infty$ as $n \to \infty$ (which can be chosen after passing to a subsequence) such that \eqref{eq:virial2} holds true:
\[
\lim_{n \to \infty} \int \sech^2 \left(\frac{x}{\lambda}\right) \left(u^2 + (\px u)^2 + \eta^2 + (\px \eta)^2\right)(t_n,x) \; dx =0.
\]
The proof is complete.
\end{proof}

\subsection{Time dependent weights} 

With small modifications of the proofs established in Section \ref{VIRIAL}, a stronger version of Proposition \ref{prop:virial} can be obtained. Let $\la=\lambda(t)$ be the time-dependent weight given by
\begin{equation}\label{eq:lambda}
\lambda(t) := \frac{C_0t}{\log^2t}, \quad t\geq 2,
\end{equation}
for any (fixed) constant $C_0>0$; precisely corresponding to the upper limit of the space interval $I(t)$ \eqref{I(t)}. The following result is stated for times $t\geq 2$, but it can be easily stated and proved for corresponding negative times.

\begin{proposition}[Decay in time-dependent intervals]\label{prop:virial2}
Let $(a,b,c)$ be \emph{dispersion-like} parameters defined as in Definition \ref{Dis_Par}. Let $(u,\eta)(t)$ be $H^1 \times H^1$ global solutions to \eqref{boussinesq} such that \eqref{Smallness} holds. Then, we have 
\begin{equation}\label{eq:virial2-1}
\int_2^{\infty}\frac{1}{\lambda(t)}\int \sech^2 \left(\frac{x}{\lambda(t)}\right) \left(u^2 + (\px u)^2 + \eta^2 + (\px \eta)^2\right)(t,x) \, dx \,dt\lesssim \ve^2.
\end{equation}
As an immediate consequence, there exists an increasing sequence of time $\{t_n\}$ $(t_n \to \infty$ as $n \to \infty)$ such that
\begin{equation}\label{eq:virial2-2}
\int \sech^2 \left(\frac{x}{\lambda(t_n)}\right) \left(u^2 + (\px u)^2 + \eta^2 + (\px \eta)^2\right)(t_n,x) \; dx \longrightarrow 0 \mbox{ as } n \to \infty. 
\end{equation}
\end{proposition} 

\begin{proof}
We choose, in functionals $\mathcal{I}, \mathcal{J}$ and $\mathcal{K}$ (see \eqref{I}-\eqref{K}),  the time-dependent weight
\begin{equation}\label{eq:vp}
\vp (t,x):= \tanh \left(\frac{x}{\lambda(t)}\right) \quad \mbox{and} \quad \vp'(t,x) := \frac{1}{\lambda(t)} \sech^2\left(\frac{x}{\lambda(t)}\right),
\end{equation}
with $\la(t)$ given by \eqref{eq:lambda}. Note that
\begin{equation}\label{eq:lambda1}
\frac{\lambda'(t)}{\lambda(t)} = \frac{1}{t}\left(1-\frac{2}{\log t} \right).
\end{equation}
Then, we compute \eqref{eq:gvirial}, but in addition we will obtain the following three new terms:
\begin{equation}\label{eq:gvirial-I}
-\frac{\lambda'(t)}{\lambda(t)} \int \frac{x}{\lambda(t)}\sech^2\left( \frac{x}{\lambda(t)}\right) \left(u\eta + \px u \px \eta\right),
\end{equation}
\begin{equation}\label{eq:gvirial-J}
-\alpha\frac{\lambda'(t)}{\lambda(t)} \int\left(1- \frac{2x}{\lambda(t)}\tanh\left( \frac{x}{\lambda(t)}\right)\right)\frac{1}{\lambda(t)}\sech^2\left( \frac{x}{\lambda(t)} \right) \eta \px u,
\end{equation}
and
\begin{equation}\label{eq:gvirial-K}
-\beta\frac{\lambda'(t)}{\lambda(t)} \int\left(1- \frac{2x}{\lambda(t)}\tanh\left( \frac{x}{\lambda(t)}\right)\right)\frac{1}{\lambda(t)}\sech^2\left( \frac{x}{\lambda(t)} \right) u\px \eta.
\end{equation}
Let $c_0 := \frac12\min(\wt{C}_1,\wt{C}_2,\wt{C}_3,\wt{C}_4)$, for $\wt{C}_i$ as in \eqref{eq:virial1-1}. In view of \eqref{eq:virial1-1} (see also \cite[Lemma 3.1]{MPP} for complete details), we have the following estimate
\begin{equation}\label{eq:virial2-3}
\begin{aligned}
|\eqref{eq:gvirial-I} &+\eqref{eq:gvirial-J} + \eqref{eq:gvirial-K}| \\
\le&~{} \frac{c_0}{\lambda(t)} \int \sech^2\Big( \frac{x}{\lambda(t)}\Big) \left(u^2 + (\px u)^2 + \eta^2 + (\px \eta)^2\right)+\frac{\wt{C}\ve^2}{t \log ^2 t},
\end{aligned}
\end{equation}
for a fixed constant $\wt{C} > 0$. This enables us to obtain a new estimate for $\frac{d}{dt} \mathcal H(t)$ of the form
\begin{equation}\label{eq:virial2-3.1}
\begin{aligned}
\frac{\wt{C}\ve^2}{t \log ^2 t} + \frac{d}{dt} \mathcal H(t)\ge&~{}  \frac{\wt{C}_1}{2\lambda(t)}\int \sech^2\Big( \frac{x}{\lambda(t)}\Big) u^2 + \frac{\wt{C}_2}{2\lambda(t)}\int \sech^2\Big( \frac{x}{\lambda(t)}\Big) (\px u)^2 \\
&+\frac{\wt{C}_3}{2\lambda(t)}\int \sech^2\Big( \frac{x}{\lambda(t)}\Big) \eta^2 + \frac{\wt{C}_4}{2\lambda(t)}\int \sech^2\Big( \frac{x}{\lambda(t)}\Big) (\px \eta)^2.
\end{aligned}
\end{equation}
Then, the same argument as in the proof of Proposition \ref{prop:virial} proves Proposition \ref{prop:virial2}, since the left-hand side of \eqref{eq:virial2-3.1} is integrable in $t$ on $[2,\infty)$, but $\la(t)^{-1}$ does not integrate in $[2,\infty)$. 
\end{proof}

\bigskip
\section{Energy estimates}\label{ENERGY}

\medskip

The proof of estimate \eqref{Conclusion_0} requires to show that the local $H^1$ norms of $(u,\eta)(t)$ converge to zero for all sequences $t_n\to +\infty$, and not only a particular one. In order to prove such a result, we will use an energy estimate, recalling that, for small solutions, the $H^1\times H^1$ norm of $(u,\eta)(t)$ squared and the energy $E[u,\eta]$ in \eqref{Energy} are equivalent.

\subsection{Preliminaries} Let $\psi = \psi(x)$ be a smooth, nonnegative and bounded function, to be chosen in the sequel. We consider the localized energy functional defined by
\begin{equation}\label{eq:local energy}
E_{loc}(t) = \frac12 \int \psi(x) \left( - a(\px u)^2 - c(\px \eta)^2 + u^2 + \eta^2 + u^2\eta\right)(t,x)dx.
\end{equation}
(Compare with \eqref{Energy}.) For the sake of simplicity, we introduce the following notations: 
\be\label{ABFG}
\begin{aligned}
&A := au_{xx}+ u +u\eta, \hspace{2em} B: = c\eta_{xx} + \eta + \frac12u^2,\\
&F := (1-\px^2)^{-1} A, \hspace{2.9em} G: = (1-\px^2)^{-1}B.
\end{aligned}
\ee
(In other words, $F$ and $G$ are the canonical variables of $A$ and $B$, respectively.) We will aso use canonical variables $f$ and $g$ for $u$ and $\eta$, and the following identities:
\begin{equation}\label{eq:E34-2}
f_{xxx}g_x = (f_{xx}g_x)_x - f_{xx}g_{xx} \hspace{1em} \mbox{and} \hspace{1em} f_xg_{xxx} = (f_xg_{xx})_x - f_{xx}g_{xx},
\end{equation}
and
\be\label{Aux_00}
f_{xx}g = (f_xg)_x - f_xg_x \hspace{1em} \mbox{and} \hspace{1em} fg_{xx} = (fg_x)_x - f_xg_x.
\ee

\subsection{Variation of local energy} We have
\begin{lemma}[Variation of local energy $E_{loc}$]\label{lem:energy1}
Let $u$ and $\eta$ satisfy \eqref{boussinesq}. Let $f$ and $g$ be canonical variables of $u$ and $\eta$ as in \eqref{eq:fg}. Then, the following hold.
\begin{enumerate}
\item Time derivative. We have
\begin{equation}\label{eq:energy1}
\begin{aligned}
\frac{d}{dt} E_{loc}(t) =&~{} \int \psi' fg + (1-2(a+c)) \int \psi' f_xg_x \\
&+ (3ac-2(a+c))\int \psi'  f_{xx}g_{xx} + 3ac\int \psi'  f_{xxx}g_{xxx}\\
& -a\int \psi''  f_xg  -c \int \psi'' fg_x\\
& + a(c-2)\int \psi''  f_{xx}g_x+c(a-2)\int \psi'' f_xg_{xx}\\
&+SNL_1(t) + SNL_2(t) + SNL_3(t) + SNL_4(t).
\end{aligned}
\end{equation}
\item The small nonlinear parts $SNL_j(t)$ are given by
\begin{equation}\label{eq:SNL1}
SNL_1(t) := \frac12a \int (\psi' u_x)_x\nlop(u^2) +c \int (\psi' \eta_x)_x\nlop(u\eta), 
\end{equation}
\begin{equation}\label{eq:SNL2}
\begin{aligned}
SNL_2(t) :=&~{} \frac12 \int \psi' f \nlop (u^2) + \frac{a}{2} \int \psi' f_{xx} \nlop (u^2)\\
&+ \int \psi' g \nlop (u \eta) + c \int \psi' g_{xx} \nlop (u \eta)\\
&+ \frac12\int \psi' f_x \nlop (u^2)_x +  \int \psi' g_x \nlop (u \eta)_x ,
\end{aligned}
\end{equation}
\begin{equation}\label{eq:SNL3}
SNL_3(t) := \frac{a}{2} \int \psi' f_{xxx} \nlop (u^2)_x+c \int \psi' g_{xxx} \nlop (u \eta)_x,
\end{equation}
and
\begin{equation}\label{eq:SNL4}
\begin{aligned}
SNL_4(t) :=&~{}\frac12 \int \psi' \nlop(u\eta) \nlop(u^2) \\
&+ \frac12 \int \psi' \nlop (u\eta)_x \nlop (u^2)_x.
\end{aligned}
\end{equation}
\end{enumerate}
\end{lemma}

\begin{proof}
We take the time derivative to $E_{loc}(t)$. The change of variable with notations $A, B, F$ and $G$ yields
\[\begin{aligned}
\frac{d}{dt} E_{loc}(t) =&~{} \int \psi \left(-au_xu_{xt} + u u_t -c\eta_x\eta_{xt} + \eta \eta_t +u\eta u_t + \frac12 u^2  \eta_t\right)\\
=&~{}\int \psi\left(u_t A + \eta_t B\right) + a\int \psi'u_t u_x + c\int \psi'\eta_t \eta_x\\
=&~{}-\int \psi \left(G_x (F-F_{xx}) + F_x (G-G_{xx}) \right) + a\int \psi'u_t u_x + c\int \psi'\eta_t \eta_x\\
=&~{}\int \psi' FG - \int \psi'F_xG_x + a\int \psi'u_tu_x + c\int \psi' \eta_t  \eta_x\\
=:&~{} E_1+E_2+E_3+E_4.
\end{aligned}\]
We first focus on $E_3$ and $E_4$. Using \eqref{eq:abcd} and the integration by parts gives
\begin{equation}\label{eq:E3}
\begin{aligned}
E_3 =&~{} a\int \psi'u_tu_x\\
=& ~{}  a\int \psi'  u_x \left( c \px \eta -(1+c)(1-\partial_x^2)^{-1}\px \eta - (1-\partial_x^2)^{-1}\px(\frac12u^2) \right)\\
=& ~{} ac\int \psi' u_x\eta_x - a(c+1)\int \psi' u_x (1-\px^2)^{-1}\eta_x\\
&+\frac12 a\int (\psi' u_x)_x(1-\px^2)^{-1}(u^2),
\end{aligned}
\end{equation}
and
\begin{equation}\label{eq:E4}
\begin{aligned}
E_4 =&~{} c\int \psi' \eta_t  \eta_x \\
=&~{} c\int \psi'   \eta_x \left( a \px u -(1+a)(1-\partial_x^2)^{-1}\px u - (1-\partial_x^2)^{-1}\px(u\eta)  \right) \\
=& ~{} ac\int \psi' \eta_x u_x - c(a+1) \int \psi' \eta_x(1-\px^2)^{-1}u_x\\
&+ c\int (\psi' \eta_x)_x(1-\px^2)^{-1}(u\eta).
\end{aligned}
\end{equation}
The last terms in both \eqref{eq:E3} and \eqref{eq:E4} corresponds to terms in  $SNL_1(t)$. For the remaining terms in both \eqref{eq:E3} and \eqref{eq:E4}, by using canonical variables $f$ and $g$ for $u$ and $\eta$, we have
\begin{equation}\label{eq:E34-1.1}
\begin{aligned}
2ac  \int \psi' u_x\eta_x  = &~ {} 2ac\int\psi'(f_xg_x + f_{xxx}g_{xxx}) \\
&~{} -2ac\int\psi'(f_{xxx}g_x + f_xg_{xxx}),
\end{aligned}
\end{equation}
\begin{equation}\label{eq:E34-1.2}
\begin{aligned}
- a(c+1)\int \psi' u_x (1-\px^2)^{-1}\eta_x= &~{}  -a(c+1)\int\psi'f_xg_x \\
&~ {} +a(c+1)\int\psi'f_{xxx}g_x,
\end{aligned}
\end{equation}
and
\begin{equation}\label{eq:E34-1.3}
\begin{aligned}
- c(a+1) \int \psi' \eta_x(1-\px^2)^{-1}u_x =&~ {}  -c(a+1)\int\psi'f_xg_x \\
&~ {} +c(a+1)\int\psi'f_xg_{xxx}.
\end{aligned}
\end{equation}
Substituting the identities \eqref{eq:E34-2} into \eqref{eq:E34-1.1}-\eqref{eq:E34-1.3}, and collecting all terms we have
\begin{equation}\label{eq:E34-3}
\begin{aligned}
E_3+ E_4 =&~{} -(a+c)\int \psi'f_xg_x + (2ac-(a+c))\int\psi'f_{xx}g_{xx} \\
&+ 2ac\int\psi'f_{xxx}g_{xxx}\\
&+a(c-1)\int\psi''f_{xx}g_x+c(a-1)\int\psi''f_xg_{xx}\\
&+SNL_1(t).
\end{aligned}
\end{equation}
\medskip

Now we deal with $E_1$ and $E_2$. From \eqref{ABFG} we have
\[
F = a\, f_{xx} + f + (1-\px^2)^{-1}(u\eta) \quad \mbox{and} \quad G = c\, g_{xx} + g + \frac12(1-\px^2)^{-1}(u^2).
\]
Using \eqref{eq:E34-2} and \eqref{Aux_00}, then a direct calculation  gives
\begin{equation}\label{eq:E1}
\begin{aligned}
E_1=\int \psi' FG =&~{} ac \int \psi'f_{xx}g_{xx} -(a+c) \int \psi'f_xg_x  + \int \psi'fg\\
&-a\int \psi''f_x g -c\int \psi'' f g_x\\ 
&+\frac{a}{2} \int \psi'f_{xx} (1-\px^2)^{-1}(u^2) +\frac12\int \psi'f(1-\px^2)^{-1}(u^2)\\
&+c\int \psi'g_{xx}(1-\px^2)^{-1}(u\eta) + \int \psi'g(1-\px^2)^{-1}(u\eta) \\
&+\frac12\int \psi'(1-\px^2)^{-1}(u\eta)(1-\px^2)^{-1}(u^2),
\end{aligned}
\end{equation}
and
\begin{equation}\label{eq:E2}
\begin{aligned}
E_2=\int \psi' F_xG_x =& ~{}ac \int \psi'f_{xxx}g_{xxx} -(a+c) \int \psi'f_{xx} g_{xx}  + \int \psi'f_xg_x\\
&-a\int \psi''f_{xx}g_x -c\int \psi''f_xg_{xx} \\
&+\frac{a}{2} \int \psi'f_{xxx} (1-\px^2)^{-1}(u^2)_x +\frac12\int \psi'f_x(1-\px^2)^{-1}(u^2)_x\\
&+c\int \psi'g_{xxx}(1-\px^2)^{-1}(u\eta)_x + \int \psi'g_x(1-\px^2)^{-1}(u\eta)_x \\
&+\frac12\int \psi'(1-\px^2)^{-1}(u\eta)_x(1-\px^2)^{-1}(u^2)_x,
\end{aligned}
\end{equation}
respectively. \eqref{eq:E34-3}, \eqref{eq:E1} and \eqref{eq:E2} prove \eqref{eq:energy1} with \eqref{eq:SNL1}-\eqref{eq:SNL4}.
\end{proof}

\begin{remark}\label{lambda_t}
The previous computations can be easily extended to the case of a weight $\psi(x)$ in \eqref{eq:local energy} depending on time, after some direct estimates on the new emergent terms are carried out. We will prove and use this fact in next Section.
\end{remark}

\bigskip

\section{Proof of the Theorem \ref{Thm1}}\label{7}

\medskip

Now we finally prove Theorem \ref{Thm1}. For the sake of simplicity in some computations, we first deal with the case of a time independent parameter $\la$, and then we extend the result to the case of $\la(t)$ given in \eqref{eq:lambda}. Let us take 
\begin{equation}\label{eq:weight1}
\psi(x) := \lambda\sech^4\left( \frac{x}{\lambda}\right) = \lambda(\vp')^2,
\end{equation}
in the localized energy $E_{loc}(t)$ \eqref{eq:local energy}, where $\vp'$ is chosen as in the proof of Proposition \ref{prop:virial}. Note that
\begin{equation}\label{eq:weight3}
|\psi'(x)| \le 4\vp' \qquad \mbox{and} \qquad |\psi''(x)| \le \frac{20}{\lambda} \vp'.
\end{equation}

Using Proposition \ref{prop:virial} and Lemmas \ref{lem:energy1}, \ref{lem:nonlinear1}, \ref{lem:nonlinear2}, \ref{lem:nonlinear3}, \ref{lem:L2 comparable} and \ref{lem:H1 comparable} carries out the following Proposition:
\begin{proposition}\label{prop:energy}
Let $(a,b,c)$ be \emph{dispersion-like} parameters defined as in Definition \ref{Dis_Par}. Let $(u,\eta)(t)$ be $H^1 \times H^1$ global solutions to \eqref{boussinesq} such that \eqref{Smallness} holds. Then, we have for large (but fixed) $\lambda \gg 1$ that
\begin{equation}\label{eq:energy2}
\lim_{t \to \infty} \int \sech^4 \left(\frac{x}{\lambda}\right) \left(u^2 + (\px u)^2 + \eta^2 + (\px \eta)^2\right)(t,x) \; dx = 0.
\end{equation}
\end{proposition}

\begin{proof}
The Cauchy-Schwarz inequality and \eqref{eq:weight3} allows that first four lines in the right-hand side of \eqref{eq:energy1} are bounded by the terms
\begin{equation}\label{eq:energy2-1}
\begin{aligned}
& \int \vp' (a_1f^2 + a_2f_x^2 + a_3f_{xx}^2 + a_4f_{xxx}^2) \\
& \qquad + \int \vp' (b_1g^2 + b_2g_x^2 + b_3g_{xx}^2 + b_4g_{xxx}^2),
\end{aligned}
\end{equation}
for some positive constants $a_j, b_j$, $j=1,2,3,4$. Coming back to standard variables, Lemmas \ref{lem:L2 comparable} and \ref{lem:H1 comparable} guarantee that 
\begin{equation}\label{eq:energy2-2}
\eqref{eq:energy2-1} \lesssim \int \vp' (u^2 + (\px u)^2 + \eta ^2 + (\px \eta)^2).
\end{equation} 
It remains to control the nonlinear terms $SNL_i(t)$, $i=1,2,3,4$, in \eqref{eq:energy1}. For $SNL_1(t)$, we use Lemma \ref{lem:nonlinear2} to obtain
\begin{equation}\label{eq:energy2-3}
|SNL_1(t)| \lesssim (\norm{u}_{H^1}+\norm{\eta}_{H^1})\int \vp' (u^2 + (\px u)^2 + \eta ^2 + (\px \eta)^2).
\end{equation}
For $SNL_2(t)$, since $f$ and $g$ are in $H^3$ (with $\norm{f}_{H^3} \lesssim \norm{u}_{H^1}$ and $\norm{f}_{H^3} \lesssim \norm{u}_{H^1}$), we have from Lemma \ref{lem:nonlinear1} that
\begin{equation}\label{eq:energy2-4}
|SNL_2(t)| \lesssim (\norm{u}_{H^1}+\norm{\eta}_{H^1})\int \vp' (u^2 + (\px u)^2 + \eta ^2 + (\px \eta)^2).
\end{equation}
For $SNL_3(t)$, we use the Sobolev embedding ($\norm{f_{xx}}_{L^{\infty}} \lesssim \norm{u}_{H^1}$) and Lemma \ref{lem:nonlinear3} to obtain
\begin{equation}\label{eq:energy2-5}
|SNL_3(t)| \lesssim (\norm{u}_{H^1}+\norm{\eta}_{H^1})\int \vp' (u^2 + (\px u)^2 + \eta ^2 + (\px \eta)^2).
\end{equation}
Lastly, for $SNL_4(t)$, note that
\[\norm{\nlop (u\eta)}_{L^{\infty}} + \norm{\nlop (u\eta)_x}_{L^{\infty}} \lesssim \norm{u\eta}_{L^2} \lesssim (\norm{u}_{H^1} + \norm{\eta}_{H^1}).\]
Lemma \ref{lem:nonlinear1} gives
\begin{equation}\label{eq:energy2-6}
|SNL_4(t)| \lesssim (\norm{u}_{H^1}+\norm{\eta}_{H^1})\int \vp' (u^2 + (\px u)^2 + \eta ^2 + (\px \eta)^2).
\end{equation}
Collecting all \eqref{eq:energy2-2}-\eqref{eq:energy2-6} we obtain
\begin{equation}\label{eq:energy2-7}
\left|\frac{d}{dt} E_{loc}(t) \right| \lesssim  \int \sech^2 \left(\frac{x}{\lambda}\right) \left(u^2 + (\px u)^2 + \eta^2 + (\px \eta)^2\right) (t,x) \; dx.
\end{equation}
Integrating on $[t, t_n]$, for $t < t_n$ as in \eqref{eq:virial2} and Proposition \ref{prop:virial} yield
\[
\begin{aligned}
\left|E_{loc}(t_n) - E_{loc}(t) \right| \lesssim &~ \int_t^{\infty}\!\! \int \sech^2 \left(\frac{x}{\lambda}\right) (u^2 + (\px u)^2 + \eta^2 + (\px \eta)^2)(t,x) \; dxdt \\
&~ < \infty.
\end{aligned}
\]
Note that, from the Sobolev embedding (with $\norm{\eta}_{H^1} \lesssim \ve$), we have 
\[
|E_{loc}(t_n)| \lesssim \int \vp'(u^2 + (\px u)^2 + \eta^2 + (\px \eta)^2) (t_n) \longrightarrow 0, \quad \mbox{ as } n \to \infty,
\]
thanks to Proposition \ref{prop:virial}. Thus, by sending $t_n \to \infty$, we have
\[\begin{aligned}
\int \sech^4\left( \frac{x}{\lambda} \right)& (u^2 + (\px u)^2 + \eta^2 + (\px \eta)^2 + \eta u^2) (t) \\
&\lesssim \int_t^{\infty}\!\!\int \sech^2 \left(\frac{x}{\lambda}\right) (u^2 + (\px u)^2 + \eta^2 + (\px \eta)^2).
\end{aligned}\]
Once again, the Sobolev embedding (with $\norm{\eta}_{H^1} \lesssim \ve$) guarantees that
\[
\lim_{t \to \infty}\int \sech^4\left( \frac{x}{\lambda} \right) (u^2 + (\px u)^2 + \eta^2 + (\px \eta)^2) (t) = 0,
\]
which completes the proof of Proposition \ref{prop:energy}.
\end{proof}

Replacing the parameter $\lambda$ by the time-dependent function $\lambda(t)$ defined in \eqref{eq:lambda}, Proposition \ref{prop:energy} extends to the case of a time-dependent, increasing interval in space. First, we take in \eqref{eq:energy1}:
\be\label{psi_psi}
\psi (t,x):= \sech^4 \Big(\frac{x}{\lambda(t)}\Big).
\ee
\begin{proposition}\label{prop:energy2}
Let $(a,b,c)$ be \emph{dispersion-like} parameters defined as in Definition \ref{Dis_Par}. Let $(u,\eta)(t)$ be $H^1 \times H^1$ global solutions to \eqref{boussinesq} such that \eqref{Smallness} holds. Then, we have 
\begin{equation}\label{eq:energy2.1}
\lim_{t \to \infty} \int \sech^4 \left(\frac{x}{\lambda(t)}\right) \left(u^2 + (\px u)^2 + \eta^2 + (\px \eta)^2\right)(t,x) \; dx = 0.
\end{equation}
\end{proposition}

\begin{remark}
This last proposition finally proves Theorem \ref{Thm1}, since
\[
\begin{aligned}
& \int \sech^4 \left(\frac{x}{\lambda(t)}\right) \left(u^2 + (\px u)^2 + \eta^2 + (\px \eta)^2\right)(t,x) \; dx \\
& \qquad \gtrsim \int_{-\la(t)}^{\la(t)} \left(u^2 + (\px u)^2 + \eta^2 + (\px \eta)^2\right)(t,x) \; dx.
\end{aligned}
\]
\end{remark}

\begin{proof}
Similarly as the proof of Proposition \ref{prop:virial2}, it suffices to deal with the following additional term in the energy estimate \eqref{eq:energy1}:
\begin{equation}\label{eq:energy2.2}
\int \psi_t \left(u^2 + (\px u)^2 + \eta^2 + (\px \eta)^2\right).
\end{equation}
Note that from \eqref{eq:lambda1} and \eqref{psi_psi},
\[
\begin{aligned}
 |\psi_t| = \left|4 \frac{\lambda'(t)}{\lambda(t)}\frac{x}{\lambda(t)}\sech^4\left(\frac{x}{\lambda(t)}\right) \tanh\left(\frac{x}{\lambda(t)}\right) \right| \lesssim \frac{1}{\lambda(t)}\sech^2\left(\frac{x}{\lambda(t)}\right).
\end{aligned}
\]
This estimate immediately implies 
\begin{equation}\label{eq:energy2.3}
\eqref{eq:energy2.2} \lesssim  \frac{1}{\lambda(t)}\int \sech^2 \left(\frac{x}{\lambda(t)}\right) (u^2 + (\px u)^2 + \eta^2 + (\px \eta)^2).
\end{equation}
Therefore, Proposition \ref{prop:virial2}, and the same argument of the proof of Proposition \ref{prop:energy}, together with estimate \eqref{eq:energy2.3} prove Proposition \ref{prop:energy2}.
\end{proof}

\appendix

\bigskip

\section{Full description of Remark \ref{Velocity_group}}\label{A}

\medskip

First of all, we recall the data given in Remark \ref{Velocity_group}. The dispersion relation is given by
\[
w(k) = \frac{\pm |k| }{1+ k^2}(1-ak^2)^{1/2}(1-ck^2)^{1/2},
\]
and the group velocity is
\[
|w'(k)|=  \frac{|ack^6 +3ack^4 -(1+2a+2c)k^2 +1|}{(1+k^2)^2 (1-ak^2)^{1/2}(1-ck^2)^{1/2}}.
\]
Since $a+c = -2 + \frac{1}{3b}$, we get 
\[
ack^6 +3ack^4 -(1+2a+2c)k^2 +1 = ack^6 +3ack^4 + \left(3- \frac{2}{3b} \right)k^2 +1.
\]
The cubic polynomial $p(\mu):=ac \mu^3 +3ac \mu^2 + (3- \frac{2}{3b} )\mu +1$ ($\mu\geq 0$, $ac>0$) never vanishes if its minimum value is positive. The derivative $p'(\mu)= 3ac \mu^2 +6ac \mu + (3- \frac{2}{3b} )$ has roots $\mu_\pm= -1 \pm (1+ \frac{2/(3b)-3}{3ac})^{1/2}$, for which the condition $\mu_\pm> 0$ does not hold if  $b\geq 2/9$. In this case readily we have $p(\mu) \geq 1$ for all $\mu\geq 0$. In consequence, $|w'(k)|>0$ for all $k\in \R$.

\bigskip

\section{Alternative expression of $\mathcal{P}$ in \eqref{eq:parameter set}.}\label{B}

\medskip

We recall from Subsection \ref{abc Conds} the alternative expression \eqref{eq:parameter set0} of the set $\mathcal{P}_0(b) = \mathcal{B}_0 \cap \mathcal{B}_1(b) \cap \mathcal{B}_2(b) \cap \mathcal{B}_3(b)$, where
\begin{equation}\label{eq:parameter set0.1}
\mathcal{P}_0(b) =~{} \left\{(a,b,c)=\left(-\frac{\nu}{2} + \frac13 -b,\; b,\; \frac{\nu}{2} - b \right) : \nu \in [0,1] \cap \left(\frac23 - 2b, 2b \right)\right\} \neq \emptyset,
\end{equation}
for $b>1/6$. Note in addition that the open interval 
\[ 
\left(\frac23 - 2b, 2b \right) \neq \emptyset
\]
for $b>1/6$. Recall the set $\mathcal P_0$ defined in \eqref{P0}:
\[
\mathcal{P}_0 = \bigcup_{b>1/6} \mathcal{P}_0(b).
\] 
Note that $\mathcal{P}_0$ is precisely the image $(a,b,c)$ of the bold set in $(b,\nu)$ in Figure \ref{Fig:0}, right panel.  Then, Lemma \ref{1o6} can be rewritten as
\begin{lemma}
There exist $(a,b,c,d)$ such that \eqref{Conds0} and  \eqref{Conds2} are satisfied if and only if $b>\frac16$. In this case, $(a,b,c) \in \mathcal{P}_0$. 
\end{lemma}
This is the starting point.  Recall again the set $\mathcal{B}_4(b)$ in Subsection \ref{Positivity} explaining \emph{dispersion-like} parameters:
\[
\mathcal{B}_4(b) := \left\{(a,c) \in \R^2  ~ : ~ 3b(a+c) +2b^2 -8ac < 0\right\}.
\]
According to \eqref{eq:parameter set0.1}, let us substitute 
\[
a = -\frac{\nu}{2} + \frac13 -b \qquad \mbox{and} \qquad c = \frac{\nu}{2} - b
\]
into $3b(a+c) +2b^2 -8ac < 0$. We get the hyperbola-like expression in $(\nu, b)$
\begin{equation}\label{eq:par_cond6}
6\left(b-\frac{11}{72}\right)^2 - \left(\nu - \frac13\right)^2 > \frac{25}{864}.
\end{equation}
Hence, $\mathcal{P}_0(b) \cap \mathcal{B}_4(b) = \widetilde{\mathcal{B}}_4(b)$, where
\[
\begin{aligned}
\widetilde{\mathcal{B}}_4(b) :=  \left\{ \left(-\frac{\nu}{2} + \frac13 -b,\; b,\; \frac{\nu}{2} - b \right) ~ : \right. & \left. ~  6\left(b-\frac{11}{72}\right)^2 - \left(\nu - \frac13\right)^2 > \frac{25}{864}, \right. \\
& \qquad  \left.\nu \in [0,1] \cap \left(\frac23 - 2b, 2b \right) \right\}.
\end{aligned}
\] 
Let us describe in more detail $\widetilde{\mathcal{B}}_4(b)$. For that, we need to consider the following two simultaneous systems of inequalities
\begin{equation}\label{eq:solve1}
\begin{cases}
\nu > \frac23 -2b\\
6\left(b-\frac{11}{72}\right)^2 - \left(\nu - \frac13\right)^2 > \frac{25}{864},
\end{cases}
\end{equation} 
and
\begin{equation}\label{eq:solve2}
\begin{cases}
\nu<2b\\
6\left(b-\frac{11}{72}\right)^2 - \left(\nu - \frac13\right)^2 > \frac{25}{864}.
\end{cases}
\end{equation} 
For  \eqref{eq:solve1} and $b>0$, we have solutions 
\begin{equation}\label{eq:solution1}
\frac13 - \sqrt{6b^2-\frac{11}{6}b +\frac19} < \nu, \qquad b <\frac14,
\end{equation}
and
\begin{equation}\label{eq:solution2}
\frac23 - 2b < \nu, \qquad b \ge \frac14.
\end{equation}
On the other hand, for \eqref{eq:solve2},
\begin{equation}\label{eq:solution3}
\nu < \frac13 + \sqrt{6b^2-\frac{11}{6}b +\frac19}, \qquad b <\frac14,
\end{equation}
and
\begin{equation}\label{eq:solution4}
\nu < 2b, \qquad b \ge \frac14.
\end{equation}
Collecting \eqref{eq:solution1}-\eqref{eq:solution4}, we get
\begin{equation}\label{eq:solution5}
\frac13 - \sqrt{6b^2-\frac{11}{6}b +\frac19} < \nu < \frac13 + \sqrt{6b^2-\frac{11}{6}b +\frac19}, \qquad b <\frac14,
\end{equation}
and
\begin{equation}\label{eq:solution6}
\frac23 -b < \nu < 2b, \qquad b \ge\frac14.
\end{equation}
Therefore, from these observation, we know that $ \widetilde{\mathcal{B}}_4 \neq \emptyset$ provided $b > \frac29$. Indeed, this is because the interval
\[
\left(\frac13 - \sqrt{6b^2-\frac{11}{6}b +\frac19}, \; \frac13 + \sqrt{6b^2-\frac{11}{6}b +\frac19}\right)
\]
is nonempty only for $b>\frac29.$ Moreover, if now
\[
I_b := [0,1] \cap \left(\frac23 - 2b, \; 2b \right) \cap \left(\frac13 - \sqrt{6b^2-\frac{11}{6}b +\frac19}, \; \frac13 + \sqrt{6b^2-\frac{11}{6}b +\frac19} \right),
\]
then we have
\[
I_b = \begin{cases}
\left(\frac13 - \sqrt{6b^2-\frac{11}{6}b +\frac19}, \frac13 + \sqrt{6b^2-\frac{11}{6}b +\frac19} \right), \qquad \frac29 < b \le \frac14,\\
\left(\frac23 - 2b, 2b \right), \qquad \frac14 < b \le \frac13,\\
\left[0, 2b \right), \qquad \frac13 < b \le \frac12, \\
[0,1], \qquad \frac12 < b.
\end{cases}
\]
Concluding, we can define
\[
\mathcal{P}(b) := \widetilde{\mathcal{B}}_4 (b)=  \left\{ \left(-\frac{\nu}{2} + \frac13 -b,\; b,\; \frac{\nu}{2} - b \right) :  \nu \in I_b\right\},
\]
for $b > \frac29$. Let
\[\mathcal{P}:=  \bigcup_{b > \frac29} \mathcal{P}(b).
\]
Then $\mathcal{P}$ describes all parameters $(a,b,c)$ for which Theorem \ref{Thm1} holds.

\providecommand{\bysame}{\leavevmode\hbox to3em{\hrulefill}\thinspace}
\providecommand{\MR}{\relax\ifhmode\unskip\space\fi MR }
\providecommand{\MRhref}[2]{%
  \href{http://www.ams.org/mathscinet-getitem?mr=#1}{#2}
}
\providecommand{\href}[2]{#2}

\end{document}